\documentclass[11pt]{amsart}

\usepackage{amscd,latexsym,amsthm,amsfonts,amssymb,amsmath,amsxtra}
\usepackage[mathscr]{eucal}
\usepackage{pictexwd,dcpic}
\usepackage[normalem]{ulem}
\usepackage{hyperref}
\usepackage{stmaryrd}

\renewcommand\theequation{\thesection.\arabic{equation}}

\newcommand{\BC}{{\mathbb {C}}}

\newcommand{\BN}{{\mathbb {N}}}

\newcommand{\BR}{{\mathbb {R}}}

\newcommand{\BZ}{{\mathbb {Z}}}

\newcommand{\CA}{{\mathcal {A}}}
\newcommand{\CB}{{\mathcal {B}}}
\newcommand{\CC}{{\mathcal {C}}}
\newcommand{\CE}{{\mathcal {E}}}
\newcommand{\CF}{{\mathcal {F}}}

\newcommand{\CK}{{\mathcal {K}}}
\newcommand{\CL}{{\mathcal {L}}}

\newcommand{\CO}{{\mathcal {O}}}
\newcommand{\CP}{{\mathcal {P}}}

\newcommand{\CR}{{\mathcal {R}}}
\newcommand{\CS}{{\mathcal {S}}}
\newcommand{\CT}{{\mathcal {T}}}
\newcommand{\CU}{{\mathcal {U}}}

\newcommand{\CX}{{\mathcal {X}}}
\newcommand{\CY}{{\mathcal {Y}}}

\newcommand{\Fa}{{\mathfrak {a}}}

\newcommand{\Fg}{{\mathfrak {g}}}
\newcommand{\Fh}{{\mathfrak {h}}}

\newcommand{\Fk}{{\mathfrak {k}}}

\newcommand{\Fn}{{\mathfrak {n}}}

\newcommand{\Fp}{{\mathfrak {p}}}

\newcommand{\Ft}{{\mathfrak {t}}}

\newcommand{\End}{{\mathrm{End}}}

\newcommand{\Hom}{{\mathrm{Hom}}}

\newcommand{\tr}{{\mathrm{tr}}}

\newcommand{\back}{\backslash}

\newtheorem{thm}{Theorem}[section]
\newtheorem{cor}[thm]{Corollary}
\newtheorem{lem}[thm]{Lemma}

\newtheorem{prop}[thm]{Proposition}

\newtheorem {ques/conj}[thm]{Question/Conjecture}

\newtheorem{defn}[thm]{Definition}
\newtheorem{rmk}[thm]{Remark}

\makeatletter

\newcommand{\Rmnum}[1]{\expandafter\@slowromancap\romannumeral #1@}
\makeatother

\begin{document}
\renewcommand{\theequation}{\arabic{equation}}
\numberwithin{equation}{section}

\title[Multiplicity formula]{A Multiplicity Formula of K-types}

\author{Chen Wan}
\address{Department of Mathematics \& Computer Science\\
Rutgers University–Newark\\
Newark, NJ 07102, USA}
\email{chen.wan@rutgers.edu}

\begin{abstract}
In this paper, by proving a simple local trace formula for real reductive groups, we prove a multiplicity formula of K-types for all irreducible representations of real reductive groups. This multiplicity formula expresses the K-characters in terms of the Harish-Chandra characters.
\end{abstract}



\maketitle

\tableofcontents

\section{Introduction}
Let $G$ be a real reductive group and $K=G^\theta$ be a maximal compact subgroup of $G$ where $\theta$ is a Cartan involution. Let $\pi$ be an irreducible representation (i.e. an irreducible Casselman-Wallach  representation) of $G$ and $\tau$ be an irreducible representation of $K$, the goal of this paper is to study the multiplicity of $K$-type
$$m(\pi,\tau)=\dim(\Hom_{K}(\pi,\tau)).$$

By Section 8.2 of \cite{Wal}, there exists a continuous linear functional $\Theta_{K,\pi}$ on $C^\infty(K)$, called the K-character, such that 
$$\Theta_{K,\pi}(f)=\sum_{\tau\in Irr(K)} m(\pi,\tau)\int_K \theta_\tau(k)f(k)dk$$
where $\theta_\tau(k)=\tr(\tau(k))$ is the character of $\tau$. In particular, we have $m(\pi,\tau)=\Theta_{K,\pi}(\bar{\theta}_\tau)$. By Theorem 8.2.2 of \cite{Wal}, we know that the K-character $\Theta_{K,\pi}$ is equal to the Harish-Chandra character $\theta_\pi$ of $\pi$ on the regular semisimple elements $K\cap G_{reg}$. In our previous paper \cite{Wan19}, we proposed a conjectural multiplicity formula for general spherical pairs. A special case of our conjectural multiplicity formula is a conjectural multiplicity formula for K-types (note that $(G,K)$ is a spherical pair). The main result of this paper is to prove this conjectural multiplicity formula of K-types. This multiplicity formula also implies a formula of the K-character in terms of the Harish-Chandra character. Roughly speaking, we will show that the K-character is only supported on certain conjugacy classes of $K$ (including all the conjugacy classes in $K\cap G_{reg}$) and its value on those conjugacy classes is equal to the average of regular germs of the Harish-Chandra character.

\subsection{Main results}
Let $\CS(G,K)$ be the set of $K$-conjugacy classes ($x\in K$ is any element in the conjugacy class) such that the pair $(G_x,K_x)$ is a minimal spherical pair. Here $G_x$ (resp. $K_x$) is the neutral component of the centralizer of $x$ in $G$ (resp. $K$). We refer the reader to Section 2.6 of \cite{Wan19} for the definition of minimal spherical pair. In the case of K-types, the pair $(G_x,K_x)$ is a minimal spherical pair if and only if $G_x$ is split modulo the center (i.e. the difference between the rank of $G_x$ and the split rank of $G_x$ is equal to the difference between the rank of $Z_{G_x}$ and the split rank of $Z_{G_x}$ where $Z_{G_x}$ is the center of $G_x$). In Section 3 we will define a measure on the set $\CS(G,K)$. Let $\theta$ be a quasi-character on $G$ (we refer the reader to Section 2 for the definition of quasi-character) and $\theta_K$ be a smooth function on $K$ that is invariant under conjugation, define
$$m_{geom,G,K}(\theta,\theta_K):=\int_{\CS(G,K)}^{\ast} \frac{1}{c(G_t,K\cap G_t,\BR)\cdot |Z_K(t):K\cap G_t|}D^{G}(t)^{1/2}\Delta(t)^{-1/2}\theta_{K}(t) c_\theta(t)dt$$
where 
\begin{itemize}
\item $c_\theta(t)$ is the average of the regular germs of $\theta$ at $t$ defined in Section \ref{section quasi-character};
\item the constant $c(G_t,K\cap G_t,\BR)$ is the number of connected components of $B_t\cap K\cap G_t$ ($B_t$ is any Borel subgroup of $G_t$) defined in Section 5 of \cite{Wan19};
\item $D^G$ (resp. $D^K$) is the Weyl determinant of $G$ (resp. $K$) and $\Delta(t)=D^G(t)D^K(t)^{-2}$;
\item $Z_K(t)$ is the centralizer of $t$ in $K$.
\end{itemize}
The integral above is not necessarily absolutely convergent and it needs to be regularized. We refer the reader to Section 3 for details. The main theorem of this paper is a multiplicity formula for K-types.

\begin{thm}\label{main thm}
For any irreducible representation $\pi$ of $G$ and for any irreducible representation $\tau$ of $K$, we have the multiplicity formula
$$m(\pi,\tau)=m_{geom,G,K}(\theta_\pi,\bar{\theta}_\tau).$$
\end{thm}

\begin{rmk}
The expression $m_{geom,G,K}(\theta_\pi,\bar{\theta}_\tau)$ is called the geometric multiplicity.
\end{rmk}

\begin{cor}
The K-character $\Theta_{K,\pi}$ is equal to
$$\Theta_{K,\pi}(f)=m_{geom,G,K}(\theta_\pi,\theta_f),\;f\in C^\infty(K)$$
where $\theta_f(x)=\int_K f(k^{-1}xk)$ is the orbital integral of $f$.
\end{cor}

\subsection{Organization of the paper and remarks on the proofs}
In this section we will briefly explain the proof of Theorem \ref{main thm}. The first step is to show that both the multiplicity and the geometric multiplicity behave nicely under parabolic induction. For the multiplicity, this is a direct consequence of the Iwasawa decomposition and the Frobenius reciprocity. For the geometric multiplicity, this is a direct consequence of Proposition \ref{germ parabolic induction}. Combine this with the Langlands classification, we only need to prove the multiplicity formula for tempered representations. This also proves the multiplicity formula for all the representations when $G$ does not have discrete series ($\iff$ $G$ does not have a maximal elliptic torus).

To prove the multiplicity formula for tempered representations, we follow the method developed by Waldspurger in his proof of the local Gan--Gross--Prasad conjecture (\cite{W10}, \cite{W12}). To be specific, we will prove a local trace formula for the K-types which would imply the multiplicity formula for tempered representations. We would like to point out that Waldspurger's method was later adapted by other people for some other cases including the unitary Gan--Gross--Prasad model (Beuzart-Plessis), the Ginzburg--Rallis model (W.), the Galois model (Beuzart-Plessis), the Shalika model (Beuzart-Plessis--W.) and some strongly tempered spherical pairs (W.--Zhang). Based on these works, the author proposed a conjectural multiplicity formula and a conjectural local trace formula for general spherical varieties in \cite{Wan19}.

Compared with all the existing cases above, there are two difficulties in the proof of the local trace formula for K-types. The first difficulty is on the spectral expansion. For all the known cases above, when the spherical pair is strongly tempered (recall that we say a spherical pair $(G,H)$ is strongly tempered if all the tempered matrix coefficients of $G$ are integrable over $H$ up to modulo the center), the proof of the spectral expansion used the fact that the multiplicity is less or equal to 1 (i.e. the spherical pair is a Gelfand pair). But this is not true for the case of K-types. To solve this issue, we combine the method in the known strongly tempered cases with the proof of the spectral expansion in the Galois model case due to Beuzart-Plessis \cite{B18} (which does not use the Gelfand pair condition but is only for discrete series instead of tempered representations). We refer the reader to Section 4 for details.

The second difficulty is on the geometric expansion. The most important step in the proof of the geometric expansion is to descend the test function to the Lie algebra and then take the Fourier transform. However, this only works when the K-type is trivial (i.e. $\theta_K\equiv 1$). In general when the K-type is nontrivial, it is not clear how to take the Fourier transform on the Lie algebra. Instead, we will first prove the geometric expansion when the $K$-type is trivial. Then we will prove the general case by using the trivial $K$-type case and the spectral side of the trace formula. We refer the reader to Section 5 and 6 for details.

This paper is organized as follows. In Section 2 we introduce basic notation and conventions used in this paper. We will also discuss some facts about quasi-characters and strongly cuspidal functions.  In Section 3 we state the multiplicity formula and the local trace formula. We also show that the trace formula implies the multiplicity formula. We will postpone the proof of a technical lemma to the Appendix. In Section 4 we will prove the spectral side of the trace formula. In Section 5 we will study the analogue of the trace formula for Lie algebra. Finally in Section 6 we will prove the geometric side of the trace formula.

\subsection{Acknowledgments}
We thank Rapha\"el Beuzart-Plessis for many helpful discussions, for the helpful comments on the first draft of this paper, and for the proof of Statement (3) in Appendix A. We also like to thank Fangyang Tian, David Vogan and Lei Zhang for many helpful conversations. The work of the author is partially supported by the NSF grant DMS-2000192 and DMS-2103720.

\section{Preliminary}\label{sec preliminary}

\subsection{Notation}\label{sec notation}
Fix a nontrivial additive character $\psi:\BR\rightarrow \BC^{\times}$. Let $G$ be a connected real reductive group, $K$ be a maximal compact subgroup of $G$, $K^\circ$ be the neutral component of $K$ which is a maximal connected compact subgroup of $G$, $\theta$ be the Cartan involution of $G$ with $K=G^\theta$, $\Fg$ be the Lie algebra of $G$ and $\Fk$ be the Lie algebra of $K^\circ$. Let $Z_G$ be the center of $G$, $A_G$ be the maximal split torus of $Z_G$. For any topology group $X$ we use $X^\circ$ to denote the neutral component of $X$.

We use $G_{ss}$, $G_{reg}$ (resp. $\Fg_{ss}$, $\Fg_{reg}$) to denote the set of semisimple and regular semisimple elements of $G$ (resp. $\Fg$) and let $\CT(G)$ (resp. $\CT_{ell}(G)$) be a set of representatives of maximal tori (resp. maximal elliptic tori) of $G$. For $x\in G_{ss}$ (resp. $X\in \Fg_{ss}$), let $Z_G(x)$ (resp. $Z_G(X)=G_X$) be the centralizer of $x$ (resp. $X$) in $G$ and let $G_x$ be the neutral component of $Z_G(x)$. Similarly, for any abelian subgroup $T$ of $G$, let $Z_G(T)$ be the centralizer of $T$ in $G$ and let $G_T$ be the neutral component of $Z_G(T)$. We say $x\in G_{ss}$ is elliptic if $x$ belongs to a maximal elliptic torus of $G$ and we use $G_{ell}\subset G_{ss}$ to denote the set of elliptic elements of $G$. We also use $G_{ell,reg}=G_{ell}\cap G_{reg}$ to denote the set of regular elliptic elements of $G$. Similarly we can define $\Fg_{ell}$ and $\Fg_{ell,reg}$. Note that the set $G_{ell,reg}$ is non-empty if and only if $G$ has discrete series. Let $G_{ss}/conj$ be the set of semisimple conjugacy classes of $G$. Finally, for $x\in G_{ss}$ (resp. $X\in \Fg_{ss}$), let $D^G(x)=|\det(1-Ad(x))_{|\Fg/\Fg_x}|$ (resp. $D^G(X)=|\det(Ad(X))_{|\Fg/\Fg_X}|$) be the Weyl determinant.

We say a subset $\Omega\subset G$ (resp. $\omega\subset \Fg$) is $G$-invariant if it is invariant under the $G$-conjugation, we say the set is completely G-invariant if it is $G$-invariant and for any $X\in \Omega$ (resp. $X\in \omega$), the semisimple part of $X$ also belongs to $\Omega$ (resp. $\omega$). For any subset $\Omega\subset G$ (resp. $\omega\subset \Fg$), we define the $G$-invariant subset
$$\Omega^G:=\{g^{-1}\gamma g\mid g\in G,\gamma\in \Omega\},\;\omega^G:=\{g^{-1}\gamma g\mid g\in G,\gamma\in \omega\}.$$
We say a $G$-invariant subset $\Omega$ of $G$ (resp. $\omega$ of $\Fg$) is compact modulo conjugation if there exists a compact subset $\Gamma$ of $G$ (resp. $\Fg$) such that $\Omega\subset \Gamma^G$ (resp. $\omega\subset \Gamma^G$). A $G$-domain on $G$ (resp. $\Fg$) is an open subset of $G$ (resp. $\Fg$) invariant under the $G$-conjugation.

We denote by $X(G)$ the group of $\BR$-rational characters of $G$. Define $\Fa_G=\Hom(X(G),\BR)$, and let $\Fa_{G}^{\ast}=X(G)\otimes_{\BZ} \BR$ be the dual of $\Fa_G$. We define a homomorphism $H_G:G\rightarrow \Fa_G$ by $H_G(g)(\chi)=\log(|\chi(g)|)$ for every $g\in G$ and $\chi\in X(G)$. This is a surjective homomorphism.

For a Levi subgroup $M$ of $G$, let $\CP(M)$ be the set of parabolic subgroups of $G$ whose Levi part is $M$, $\CL(M)$ be the set of Levi subgroups of $G$ containing $M$, and $\CF(M)$ be the set of parabolic subgroups of $G$ containing $M$. We have a natural decomposition $\Fa_M=\Fa_{M}^{G}\oplus \Fa_G$. We denote by $proj_{M}^{G}$ and $proj_G$ the projections of $\Fa_M$ to each factors. For each $P\in \CP(M)$, we can associate a positive chamber $\Fa_{P}^{+}\subset \Fa_M$. For each $P=MU$, we can also define a function $H_P:G\rightarrow \Fa_M$ by $H_P(g)=H_M(m_g)$ where $g=m_g u_g k_g$ is the Iwasawa decomposition of $g$.

In this paper we shall freely use the notion of log-norms on algebraic varieties  as defined in Section 1.2 of \cite{B15}. For every algebraic variety $X$ over $\BR$, we will fix a log-norm $\sigma_X$ on it. For $C>0$, we use $X[<C]$ to denote the set $\{x\in X|\;\sigma_X(x)<C\}$. In particular, we have log-norms $\sigma_G$ and $\sigma_{\Fg}$ on $G$ and $\Fg$ respectively. It will be convenient to assume, as we may, that $\sigma_G$ is left and right $K$-invariant. Following Harish-Chandra, we can also define the height function $\Vert \cdot\Vert_G$ on $G$ (resp. $\Vert \cdot\Vert_\Fg$ on $\Fg$), taking values in $\BR_{\geq 0}$ so that the log-norm $\sigma_G$ on $G$ (resp. $\sigma_\Fg$ on $\Fg$) is given by $\sigma_G(g)=\sup(1,\log(\Vert g\Vert_G))$ (resp. $\sigma_\Fg(X)=\sup(1,\log(\Vert X\Vert_\Fg))$).

We fix a minimal Levi subgroup (resp. parabolic subgroup) $M_0$ (resp. $P_0=M_0N_0$) of $G$ and let $A$ be a maximal split torus of $M_0$. Let $\Sigma(A,P_0)$ be the set of roots of $A$ in $P_0$ and let $A^+=\{a\in A^\circ|\;\alpha(a)\geq 1,\;\forall \alpha\in \Sigma(A,P_0)\}$. We have the Cartan decomposition $G=KA^+K$.

We say a parabolic subgroup of $G$ is standard if it contains $P_0$. We say a Levi subgroup of $G$ is standard if it is a Levi subgroup of a standard parabolic subgroup and it contains $M_0$.

For two complex valued functions $f$ and $g$ on a set $X$ with $g$ taking values in $\BR_{\geq 0}$, we write that
$$
f(x)\ll g(x)
$$
and say that $f$ is essentially bounded by $g$, if there exists a constant $c>0$ such that for all $x\in X$, we have
$$
| f(x)| \leq cg(x).
$$
We say $f$ and $g$ are equivalent, which is denoted by
$$f(x)\sim g(x)$$
if $f$ is essentially bounded by $g$ and $g$ is essentially bounded by $f$.

Let $C^{\infty}(G)$ be the space of smooth functions on $G$ and let $C_{c}^{\infty}(G)$ (resp. $\CS(G)$) be the space of smooth compactly supported functions (resp. Schwartz functions) on $G$. We use $\CC(G)$ to denote the Harish-Chandra-Schwartz space of $G$ (see Section 1.5 of \cite{B15} for details). On the Lie algebra level, let $C_{c}^{\infty}(\Fg)$ (resp. $\CS(\Fg)$) be the space of smooth compactly supported functions (resp. Schwartz functions) on $\Fg$. We also use $\Xi^G$ to denote the Harish-Chandra $\Xi$-function of $G$.

Let $C_{c,scusp}^{\infty}(G)\subset C_{c}^{\infty}(G)$ be the subspace of strongly cuspidal functions in $C_{c}^{\infty}(G)$. Similarly we can define the spaces $\CS_{scusp}(G), \CC_{scusp}(G),$ $C_{c,scusp}^{\infty}(\Fg),$ and $\CS_{scusp}(\Fg)$. We refer the reader to Section 5 of \cite{B15} for the definition and basic properties of strongly cuspidal functions.

\subsection{Measure}\label{sec measure}
If $G$ is a connected reductive group, we may fix a non-degenerate symmetric bilinear form $<\cdot,\cdot>$ on $\Fg$ that is invariant under $G$-conjugation (i.e. Killing form). For $f\in \CS(\Fg)$, we can define its Fourier transform $f\rightarrow \hat{f}$ to be
\begin{equation}\label{FT}
\hat{f}(X)=\int_{\Fg} f(Y) \psi(<X,Y>) dY
\end{equation}
where $dY$ is the selfdual Haar measure on $\Fg$ such that $\hat{\hat{f}}(X)=f(-X)$. Then we get a Haar measure on $G$ such that the Jacobian of the exponential map is equal to 1 at $0\in \Fg$. If $H$ is a subgroup of $G$ such that the restriction of the bilinear form to $\Fh$ is also non-degenerate, then we can define the measures on $\Fh$ and $H$ by the same method.

Let $G$ be a split real reductive group with trivial center, $K=G^\theta$ be a maximal compact subgroup of $G$, $\Fk$ be the Lie algebra of $K^\circ$ and $\Fp=\Fk^\perp$ be the orthogonal complement of $\Fk$ in $\Fg$ (i.e. $\Fg=\Fp\oplus \Fk$ is the Cartan decomposition). Let $\Ft\subset \Fp$ be a maximal abelian subspace. Since $G$ is split, we know that $\Ft$ is the Lie algebra of a maximal split torus $T$ of $G$. Let $B=TN$ be a Borel subgroup of $g$ and $\bar{B}=T\bar{N}$ be its opposite. We have the Iwasawa decomposition $G=BK=\bar{B}K$ and $T\cap K$ is a finite group. The next lemma will be used in the proof of the geometric side of the trace formula.

\begin{lem}\label{Cartan decomposition}
For $f\in \CS(\Fp)$, we have ($W(T)$ is the Weyl group)
$$\int_{\Fp} f(X)dX=\frac{1}{|W(T)|\cdot |T\cap K|}\int_{\Ft}\int_K D^G(Y)^{1/2}f(k^{-1}Yk)dkdY.$$
\end{lem}

\begin{proof}
The lemma is a consequence of the following three facts.
\begin{itemize}
\item Every element in $\Fp\cap \Fg_{reg}$ is $K$-conjugated to an element of $\Ft$.
\item Two elements in $\Fp\cap \Fg_{reg}$ are $G$-conjugated to each other if and only if they are $K$-conjugated to each other.
\item The Jacobian of the map $\Fp//K\rightarrow \Ft//W(T)$ is equal to $D^G(\cdot)^{1/2}$.
\end{itemize}
The first fact follows from the fact that two different choices of maximal abelian subspace of $\Fp$ are $K$-conjugated to each other. The second fact follows from the first fact and the fact that the normalizer of $T$ in $G$ is contained in $TK$. 

For the third one, let $\Sigma$ (resp. $\bar{\Sigma}$) be the roots of $T$ in $\Fn=Lie(N)$ (resp. $\bar{\Fn}=Lie(\bar{N})$). Then $\Sigma\cup \bar{\Sigma}$ is the root system of $G$ with $\Sigma$ (resp. $\bar{\Sigma}$) being the set of positive (resp. negative) roots. For $\alpha\in \Sigma\cup \bar{\Sigma}$, let $V_\alpha$ be the root space of $\alpha$. Since $G$ is split, we know that $V_\alpha$ is one dimensional and the Cartan involution $\theta$ will map $V_\alpha$ onto $V_{-\alpha}$ for all $\alpha\in \Sigma\cup \bar{\Sigma}$. Hence we can fix $0\neq X_\alpha\in V_\alpha$ such that $\theta(X_\alpha)=X_{-\alpha}$ for all $\alpha\in \Sigma\cup \bar{\Sigma}$. Then $\{X_\alpha+X_{-\alpha}|\;\alpha\in \Sigma\}$ is a basis of $\Fk$ and $\Fp=\Ft\oplus Span\{X_\alpha-X_{-\alpha}|\;\alpha\in \Sigma\}$. This implies that the Jacobian is equal to $D^G(\cdot)^{1/2}$. This finishes the proof of the lemma.
\end{proof}

\subsection{Representations}\label{sec representation}
We say a representation $\pi$ of $G$ is irreducible  (resp. finite length) if it is an irreducible (resp. finite length) Casselman-Wallach representation of $G$. We say a finite length representation $\pi$ of $G$ is an induced representation if there exists a proper parabolic subgroup $P=MN$ of $G$ and a finite length  representation $\tau$ of $M$ such that $\pi=I_{P}^{G}(\tau)$. Here $I_{P}^{G}(\cdot)$ is the normalized parabolic induction.

We use $\CR(G)$ to denote the Grothendieck group of finite length  representations of $G$, and we use $\CR(G)_{ind}\subset \CR(G)$ (resp. $\CR(G)_{temp}\subset \CR(G)$) to denote the subspace of $\CR(G)$ generated by induced representations (resp. tempered representations).

\begin{prop}\label{Grothendiect group}
We have $\CR(G)=\CR(G)_{ind}+\CR(G)_{temp}$ (i.e. $\CR(G)$ is generated by induced representations and tempered representations). Moreover, if $G_{ell,reg}=\emptyset$, then $\CR(G)=\CR(G)_{ind}$.
\end{prop}

\begin{proof}
The first part is a direction consequence of the Langlands classification. For the second part, since $G_{ell,reg}=\emptyset$, $G$ does not have any elliptic representations. This implies that all the tempered representations of $G$ are generated by induced representations. This proves the proposition.
\end{proof}

\subsection{Quasi characters}\label{section quasi-character}
Let $Nil(\Fg)$ be the set of nilpotent orbits  of $\Fg$ and $Nil_{reg}(\Fg)$ be the set of regular nilpotent orbits of $\Fg$. In particular, the set $Nil_{reg}(\Fg)$ is empty unless $G$ is quasi-split. For every $\CO\in Nil(\Fg)$ and $f\in \CS(\Fg)$, we use $J_{\CO}(f)$ to denote the nilpotent orbital integral of $f$ associated to $\CO$. Harish-Chandra proved that there exists a unique smooth function $Y\rightarrow \hat{j}(\CO,Y)$ on $\Fg_{reg}$,  which is invariant under $G$-conjugation and locally integrable on $\Fg$, such that for every $f\in \CS(\Fg)$, we have
$$J_{\CO}(\hat{f})=\int_{\Fg} f(Y) \hat{j}(\CO,Y) dY.$$

We refer the reader to Section 4.2-4.4 of \cite{B15} for the definition of quasi-characters. Let $\theta$ (resp. $\theta'$) be a quasi character on $G$ (resp. $\Fg$), we have the germ expansions 
\begin{eqnarray*}
&&D^G(x\exp(X))^{1/2}\theta(x\exp(X))\\
&=&D^G(x\exp(X))^{1/2}\sum_{\CO\in Nil_{reg}(\Fg_x)} c_{\theta,\CO}(x) \hat{j}(\CO,X)+O(|X|),\\
&&D^G(X+Y)^{1/2}\theta'(X+Y)\\
&=&D^G(X+Y)^{1/2}\sum_{\CO\in Nil_{reg}(\Fg_X)} c_{\theta',\CO}(X) \hat{j}(\CO,Y)+O(|Y|)
\end{eqnarray*}
for every $x\in G_{ss}$ (resp. $X\in \Fg_{ss}$) and $X\in \Fg_x$ (resp. $Y\in \Fg_X$) close to 0. Here $c_{\theta,\CO}(x)\in \BC$ (resp. $c_{\theta',\CO}(X)\in \BC$) are called the regular germs of $\theta$ (resp. $\theta'$) at $x$ (resp. $X$).

The most important examples of quasi-characters on $G$ are the Harish-Chandra characters of finite length smooth representations of $G$. Examples of quasi-characters on $\Fg$ are the functions $\hat{j}(\CO,\cdot)$ ($\CO\in Nil(\Fg)$) defined above.

We use $QC(G)$ (resp. $QC(\Fg)$) to denote the set of quasi-characters on $G$ (resp. $\Fg$), and we use $QC_c(G)$ (resp. $QC_c(\Fg)$) to denote the set of quasi-characters on $G$ (resp. $\Fg$) whose support is compact modulo conjugation. We also use $SQC(\Fg)$ to denote the space of Schwartz quasi-character on $\Fg$ defined in Section 4.2 of \cite{B15}. We have $QC_c(\Fg)\subset SQC(\Fg)$. If $\Omega$ (resp. $\omega$) is an open completely G-invariant subset of $G$ (resp. $\Fg$), we define $QC(\Omega)$ (resp. $QC(\omega)$) to be the set of quasi-characters on $\Omega$ (resp. $\omega$). Similarly we can also define the spaces $QC_c(\Omega),\;QC_c(\omega)$. We refer the reader to Section 4.2-4.4 of \cite{B15} for the topology on these spaces.

For $f\in \CC_{scusp}(G)$ (resp. $f\in \CS_{scusp}(\Fg)$), let $\theta_f$ be the quasi-character on $G$ (resp. $\Fg$)
defined via the weighted orbital integrals of $f$. For $f\in \CS_{scusp}(\Fg)$, let $\hat{\theta}_f=\theta_{\hat{f}}$ be the Fourier transform of $\theta_f$. We refer the reader to Section 5.2 and 5.6 of \cite{B15} for details.

\begin{defn}
Let $\theta$ be a quasi-character on $G$. For $x\in G_{ss}$, define the average of the regular germs to be
$$c_{\theta}(x)=\begin{array}{cc}\left\{ \begin{array}{ccl} \frac{1}{|Nil_{reg}(\Fg_x)|} \sum_{\CO\in Nil_{reg}(\Fg_x)} c_{\theta,\CO}(x) & \text{if} & Nil_{reg}(\Fg_x)\neq \emptyset; \\ 0 & \text{if} & Nil_{reg}(\Fg_x)=\emptyset. \\ \end{array}\right. \end{array}$$
\end{defn}

\begin{rmk}
\begin{enumerate}
\item \emph{The set $ Nil_{reg}(\Fg_x)$ is non-empty if and only if $G_x$ is quasi-split.}
\item \emph{For $x\in G_{reg}$, $c_{\theta}(x)$ is just $\theta(x)$.}
\end{enumerate}
\end{rmk}

\begin{defn}
For $\lambda\in \BR^{\times}$, $l\in \BR$ and $\theta\in QC(\Fg)$, define $$M_{\lambda,l}(\theta)(X)=|\lambda|^{-l} \theta(\lambda^{-1}X),\;X\in \Fg.$$
We also use $\theta_\lambda$ to denote $M_{\lambda,0}(\theta)$, i.e. $\theta_\lambda(X)=\theta(\lambda^{-1} X)$.
\end{defn}

\begin{lem}\label{lemma quasi-character homogeneous}
For $\lambda\neq \pm 1$, $l\in \BZ_{>0}$ and $\theta\in QC_c(\Fg)$, the following hold.
\begin{enumerate}
\item If $c_{\theta,\CO}=0$ for all $\CO\in Nil_{reg}(\Fg)$, then there exists $\theta_1,\theta_2\in QC_c(\Fg)$ such that $\theta=(M_{\lambda,\delta(G)/2}-1)^d\theta_1+\theta_2$ and $0\notin Supp(\theta_2)$. Here $\delta(G)=\dim(G)-rank(G)$.
\item For $l>\delta(G)/2$,  there exists $\theta_1,\theta_2\in QC_c(\Fg)$ such that $\theta=(M_{\lambda,l}-1)^d\theta_1+\theta_2$ and $0\notin Supp(\theta_2)$.
\end{enumerate}
\end{lem}

\begin{proof}
The first part is just Proposition 4.6.1(i) of \cite{B15}. The second part follows from the same proof as in loc. cit. together with the fact that $(D^G)^{1/2}\hat{j}(\CO,\cdot)$ is locally bounded for all $\CO\in Nil(\Fg)$ (Theorem 17 of \cite{V}).
\end{proof}

\begin{rmk}
If we assume that $0<\lambda<1$, then the above lemma is also true when we replace $QC_c(\Fg)$ by $QC_c(\omega)$ for any complete $G$-invariant convex neighborhood $\omega$ of $0$ in $\Fg$. Here we say $\omega$ is convex if for any $X\in \omega$ and $0\leq \lambda\leq 1$, we have $\lambda X\in \omega$.
\end{rmk}

Let $P=MN$ be a parabolic subgroup of $G$, $\theta_M$ be a quasi-character of $M$ and $\theta=I_{P}^{G}(\theta_M)$ (we refer the reader to Section 3.4 and 4.7 of \cite{B15} for the definition of parabolic induction of quasi-characters). For all $x\in G_{ss}$, let $\CX_M(x)$ be a set of representatives for the $M$-conjugacy classes of elements in $M$ that are $G$-conjugated to $x$. The following proposition was proved in Proposition 4.7.1 of \cite{B15} and it tells us the behavior of $c_{\theta}(x)$ under parabolic induction.

\begin{prop}\label{germ parabolic induction}
For all $x\in G_{ss}$, $D^G(x)^{1/2}c_{\theta}(x)$ is equal to
$$|Z_G(x):G_x| \sum_{y\in \CX_M(x)} |Z_M(y):M_y|^{-1} D^M(y)^{1/2} c_{\theta_M}(y).$$
In particular, $c_{\theta}(x)=0$ if the set $\CX_M(x)$ is empty.
\end{prop}

The next lemma is well known (e.g. Lemma 3.2 of \cite{ABV}) and it will be used in our proof of the geometric side of the trace formula.

\begin{lem}\label{lemma nilpotent orbit}
Let $G$ be a quasi-split real reductive group and $\CO_1,\CO_2$ be two regular nilpotent orbits of $\Fg$. Then there exists $g\in Res_{\BC/\BR}G$ such that the $g$-conjugation map preserves $G$ and it sends $\CO_1$ to $\CO_2$.
\end{lem}

\begin{rmk}
Since any two maximal compact subgroups of $G$ are conjugated to each other, we may choose $g$ in the above lemma such that $g^{-1}Kg=K$.
\end{rmk}

\section{The trace formula and the multiplicity formula}\label{sec trace formula and multiplicity formula}

\subsection{The distribution $I$}\label{sec distribution I}
Let $\theta_K$ be a smooth function on $K$ that is invariant under conjugation. For $f\in \CC(G)$, define
$$I(f,x,\theta_K)=\int_{K} f(x^{-1}kx)\theta_K(k) dk.$$
Set $||\theta_K||=\max_{k\in K} |\theta_K(k)|$. Let $\Theta_K$ be a set consists of some smooth $K$-invariant functions on $K$ with a fixed upper on both the functions and derivatives, i.e. there exists $C>0$ such that 
$$||\theta_K||<C,\;|\frac{d}{dt} \theta_K(k\exp(tX))|<C\cdot ||X||_{\Fk}$$
for all $\theta_K\in \Theta_K,k\in K$ and $X\in \Fk$. Our goal is to prove the following proposition.

\begin{prop}\label{convergence}
For $d>0$, there exists a norm $\nu_d$ on $\CC(G)$ such that 
$$|I(f,x,\theta_K)|\leq \nu_d(f)(||\theta_K||+\sqrt{||\theta_K||})\Xi^G(x)^2\sigma_{G/Z_G}(x)^{-d}$$
for all $\theta_K\in \Theta_K$, $x\in G$ and $f\in \CC_{scusp}(G)$. Here $\Xi^G$ is the Harish-Chandra $\Xi$-function on $G$.
\end{prop}

The proof is very similar to Theorem 8.1.1 of \cite{B15}. By the Cartan decomposition $G=KA^+K$, it is enough to consider the case when $x=a\in A^+$. For every standard parabolic subgroup $Q=LU_Q$ of $G$ with $A\subset L$ and $\delta>0$, set ($R(A,U_Q)$ is the set of roots of $A$ in $U_Q$)
$$A^{Q,+}(\delta)=\{a\in A^+|\;|\alpha(a)|\geq e^{\delta \cdot \sigma_{G/Z_G}(a)},\;\forall \alpha\in R(A,U_Q)\}.$$
We choose $\delta>0$ small so that the complement of 
$$\cup_{Q\;\text{standard}} A^{Q,+}(\delta)$$
in $A^+$ is compact modulo the center $Z_G$. Hence we only need to prove the estimate for $a\in A^{Q,+}(\delta)$. 

Let $\bar{Q}=L\bar{U}_Q$ be the opposite parabolic subgroup of $Q$, $K_Q=K\cap Q$. Up to conjugating $L$ by some element in $U_Q$ we may assume that $K_Q\subset L$ and we let $K_L=K_Q$. We can define the function $\theta_{K,L}$ on $K_L$ to be $\theta_{K,L}=\theta_K|_{K_L}$.

Set $K^Q=K_L\ltimes \bar{U}_Q$ and define $\theta_K^Q$ on $K^Q$ to be $\theta_K^Q(k_Lu_Q)=\theta_{K,L}(k_L)$. We fix the Haar measures on $K^Q,K_Q\simeq K_L$ and $\bar{U}_Q$ so that
$$\int_{K^Q}f(k^Q)dk^Q=\int_{K_L}\int_{\bar{U}_Q} f(k_Lu_Q)  du_Q dk_L.$$

For $f\in \CC(G)$ and $a\in A^+$, we define
$$I^Q(f,a,\theta_K)=\int_{K^Q} f(a^{-1}k^Qa)\theta_K^Q(k^Q)dk^Q.$$
Since $I^Q(f,a,\theta_K)=0$ if $f$ is strongly cuspidal, we are reduced to prove the following proposition.

\begin{prop}
There exists a constant $c>0$ depends on the choice of Haar measure such that for all $d>0$, we have
$$|I(f,a,\theta_K)-cI^Q(f,a,\theta_K)|\leq \nu_d(f)(||\theta_K||+\sqrt{||\theta_K||})\Xi^G(a)^2\sigma_{G/Z_G}(a)^{-d}$$
for all $\theta_K\in \Theta_K$, $a\in A^{Q,+}(\delta)$ and $f\in \CC(G)$. Here $\nu_d$ is a norm on $\CC(G)$ depends on $d$.
\end{prop}

\begin{proof}
Let $K'=K\cap P_0\bar{N}_0$ and let $u:K'\rightarrow \bar{N}_0$ be the map sending $k$ to the unique element $u(k)$ in $\bar{N}_0$ so that $ku(k)^{-1}\in P_0$. Recall that $P_0=M_0N_0$ is a minimal parabolic subgroup of $G$ and $\bar{P}_0=M_0\bar{N}_0$ is its opposite parabolic subgroup. By the Iwasawa decomposition we know this map is submersive at the identity. Hence we can find a relatively compact open neighborhood $\CU$ of $1$ in $\bar{N}_0$ and an $\BR$-analytic section 
$$k:\CU\rightarrow K',\;u\mapsto k(u)$$
to the map $u(\cdot)$ over $\CU$ with $k(1)=1$. Set $\CU_Q=\CU\cap \bar{U}_Q$ and $\CK=K_Qk(\CU_Q)$. By the same argument as in (8.1.7) of \cite{B15}, we have
\begin{itemize}
\item[(1)] The map $\iota: K_Q\times \CU_Q\rightarrow K,\;(k_Q,u)\mapsto k_Qk(u)$ is an $\BR$-analytic embedding with image $\CK$ and there exists a smooth function $j$ on $\CU_Q$ such that 
$$\int_{\CK}\varphi(k)dk=\int_{K_Q}\int_{\CU_Q} \varphi(k_Q k(u)) j(u)dudk_Q$$
for all $\varphi\in L^1(\CK)$.
\end{itemize}

By Lemma 1.3.1(ii) of \cite{B15}, for $\epsilon>0$ small enough, we have
$$a\bar{U}_Q[<\epsilon \sigma_{G/Z_G}(a)] a^{-1}\subset \CU_Q$$
for all $a\in A^{Q,+}(\delta)$. We set
$$K^{<\epsilon,a}=K_Q k(a\bar{U}_Q[<\epsilon \sigma_{G/Z_G}(a)] a^{-1}),\; K^{Q,<\epsilon,a}=K_L a\bar{U}_Q[<\epsilon \sigma_{G/Z_G}(a)] a^{-1},$$
$$I^{<\epsilon}(f,a,\theta_K)=\int_{K^{<\epsilon,a}} f(a^{-1}ka)\theta_K(k)dk,\;I^{Q,<\epsilon}(f,a,\theta_K)=\int_{K^{Q,<\epsilon,a}} f(a^{-1}k^Qa)\theta_K^Q(k^Q) dk^Q$$
for $a\in A^{Q,+}(\delta)$. Let $0<\delta_0<\delta/2$ and $c=j(1)$. We only need to prove the following 3 inequalities ($\nu_d$ is a norm on $\CC(G)$ depends on $d$)
\begin{equation}\label{1}
|I(f,a,\theta_K)-I^{<\epsilon}(f,a,\theta_K)|\leq \nu_d(f) ||\theta_K||\Xi^G(a)^2\sigma_{G/Z_G}(a)^{-d},    
\end{equation}
\begin{equation}\label{2}
|I^{Q}(f,a,\theta_K)-I^{Q,<\epsilon}(f,a,\theta_K)|\leq \nu_d(f) ||\theta_K||\Xi^G(a)^2\sigma_{G/Z_G}(a)^{-d},
\end{equation}
\begin{equation}\label{3}
|I^{<\epsilon}(f,a,\theta_K)-cI^{Q,<\epsilon}(f,a,\theta_K)|<\nu_{d}(f) \sqrt{||\theta_K||}\Xi^G(a)^2 e^{-\delta_0\cdot \sigma_{G/Z_G}(a)}
\end{equation}
for all $\theta_K\in \Theta_K$, $a\in A^{Q,+}(\delta)$ and $f\in \CC(G)$. 

\begin{rmk}
The three inequalities above is an analogue of (8.1.8)-(8.1.10) of \cite{B15}. In our case, since $K_Q=K_L$ is compact, we only need to use $K_Q,K_L$ instead of $K_Q[<\epsilon \sigma_{G/Z_G}(a)],K_L[<\epsilon \sigma_{G/Z_G}(a)]$ in the definition of  $K^{<\epsilon,a},K^{Q,<\epsilon,a}$. 
\end{rmk}

By the same argument as in (8.1.11) of \cite{B15} (note that Proposition 6.4.1(iii) and Proposition 6.8.1(v) of loc. cit. are trivial for our case since $K$ is compact), we have
\begin{equation}
\sigma_{G/Z_G}(a)\ll \sigma_{G/Z_G}(a^{-1}k^Qa),\;\sigma_{G/Z_G}(a)\ll \sigma_{G/Z_G}(a^{-1}ka)
\end{equation}
for all $a\in A^{Q,+}(\delta)$, $k\in K-K^{<\epsilon,a}$, and $k^Q\in K^Q-K^{Q,<\epsilon,a}$. As a result, we have
$$|I(f,a,\theta_K)-I^{<\epsilon}(f,a,\theta_K)|\leq \int_{K-K^{<\epsilon,a}} |f(a^{-1}ka)\theta_K(k)|dk$$
$$\leq \nu_d(f)||\theta_K||\cdot \int_{K-K^{<\epsilon,a}} \Xi^G(a^{-1}ka)\sigma_G(a^{-1}ka)^{-d}dk\ll \nu_d(f)||\theta_K||\sigma_{G/Z_G}(a)^{-d}\cdot \int_{K}\Xi^G(a^{-1}ka)dk$$
$$=\nu_d(f)||\theta_K||\sigma_{G/Z_G}(a)^{-d}\Xi^G(a^{-1})\Xi^G(a)\sim \nu_d(f)||\theta_K||\sigma_{G/Z_G}(a)^{-d}\Xi^G(a)^2,$$
and
$$|I^{Q}(f,a,\theta_K)-I^{Q,<\epsilon}(f,a,\theta_K)|\leq \int_{K^Q-K^{Q,<\epsilon,a}} |f(a^{-1}ka)\theta_K(k)|dk$$
$$\leq \nu_d(f)||\theta_K||
\cdot \int_{K^Q-K^{Q,<\epsilon,a}} \Xi^G(a^{-1}ka)\sigma_G(a^{-1}ka)^{-3d}dk$$
$$\ll \nu_d(f)||\theta_K||\sigma_{G/Z_G}(a)^{-2d}\cdot \int_{K^Q}\Xi^G(a^{-1}ka)\sigma_G(a^{-1}ka)^{-d}dk$$
$$=\nu_d(f)||\theta_K||\sigma_{G/Z_G}(a)^{-2d}\delta_{Q}(a)^{-1}\cdot \int_{K_L}\int_{\bar{U}_Q}\Xi^G(a^{-1}kau)\sigma_G(a^{-1}kau)^{-d}dk$$
$$\ll \nu_d(f)||\theta_K||\sigma_{G/Z_G}(a)^{-2d}\delta_{Q}(a)^{-1}\cdot \int_{K_L}\Xi^L(a^{-1}ka)dk$$
$$\sim \nu_d(f)||\theta_K||\sigma_{G/Z_G}(a)^{-2d}\delta_{Q}(a)^{-1}\Xi^L(a^{-1})\Xi^L(a)\ll \nu_d(f)||\theta_K||\sigma_{G/Z_G}(a)^{-d}\Xi^G(a)^2$$
for all $\theta_K\in \Theta_K$, $a\in A^{Q,+}(\delta)$ and $f\in \CC(G)$. Here we have used Proposition 1.5.1 (i), (iv) and (vi) of \cite{B15}. The norm $\nu_d$ is the above inequalities can be chosen to be $\nu_d(f)=\sup_{g\in G} |\frac{f(g)}{\Xi^G(g)\sigma_G(g)^{-3d}}|$. This proves \eqref{1} and \eqref{2}.

For \eqref{3}, we have
$$I^{<\epsilon}(f,a,\theta_K)=\int_{K_L} \int_{a\bar{U}_Q[<\epsilon \sigma_{G/Z_G}(a)]a^{-1}} f(a^{-1}k_L k(u_Q)a)\theta_K(k_L k(u_Q)) j(u_Q)du_Q dk_Q,$$
$$I^{Q,<\epsilon}(f,a,\theta_K)=\int_{K_L} \int_{a\bar{U}_Q[<\epsilon \sigma_{G/Z_G}(a)]a^{-1}} f(a^{-1}k_L u_Q a)\theta_K(k_L) du_Q dk_L.$$
Fix $\delta'>0,d'>0$ with $0<2\delta_0<\delta'<\delta$ and $d'>0$ large. We first show that in order to prove \eqref{3}, it is enough to prove the following statement.

\begin{itemize}
\item[(2)] For $\epsilon>0$ small, we have
$$|f(a^{-1}k_L k(u_Q)a)\theta_K(k_L k(u_Q)) j(u_Q)-cf(a^{-1}k^Qa)\theta_K(k_L)|$$
$$< \nu_d(f)\sqrt{||\theta_K||}\Xi^G(a^{-1}k^Qa)\sigma_{G}(a^{-1}k^Qa)^{-d'}e^{-\frac{\delta'}{2}\cdot \sigma_{G/Z_G}(a)}$$
for all $\theta_K\in \Theta_K$, $a\in A^{Q,+}(\delta)$, $u_Q\in a\bar{U}_Q[<\epsilon \sigma_{G/Z_G}(a)]a^{-1}$, $k_L\in K_L$, $f\in \CC(G)$, and for some norm $\nu_d$ on $\CC(G)$. Here we have set $k^Q=k_Lu_Q$. 
\end{itemize}

In fact, if (2) holds, then we have
$$|I^{<\epsilon}(f,a,\theta_K)-cI^{Q,<\epsilon}(f,a,\theta_K)|< \nu_d(f)\sqrt{||\theta_K||}e^{-\frac{\delta'}{2}\cdot \sigma_{G/Z_G}(a)}\int_{K^{Q,<\epsilon,a}}\Xi^G(a^{-1}k^Qa)\sigma_{G}(a^{-1}k^Qa)^{-d'}dk^Q$$
$$\leq \nu_d(f)\sqrt{||\theta_K||}e^{-\frac{\delta'}{2}\cdot \sigma_{G/Z_G}(a)} \int_{K^Q} \Xi^G(a^{-1}k^Qa)\sigma_{G}(a^{-1}k^Qa)^{-d'}dk$$
$$\ll \nu_d(f)\sqrt{||\theta_K||}e^{-\frac{\delta'}{2}\cdot \sigma_{G/Z_G}(a)}\Xi^G(a^{-1})\Xi^G(a)\sim \nu_d(f)\sqrt{||\theta_K||}e^{-\frac{\delta'}{2}\cdot \sigma_{G/Z_G}(a)}\Xi^G(a)^2$$
for all $\theta_K\in \Theta_K$, $a\in A^{Q,+}(\delta)$ and $f\in \CC(G/A_{G}^{\circ})$. Here the inequality in the third line follows from Proposition 1.5.1(iv) and (vi) of \cite{B15}. This proves \eqref{3}.

To prove (2), first by Lemma 1.3.1(ii) of \cite{B15} together with the facts that $k_L$ belongs to a compact set and $||\theta_K'||,||\theta_K||$ are bounded, we know that there exists a constant $C$ that depends on the upper bounds of $||\theta_K||,\;||\theta_K'||,\;||j||,\;||j'||$ (in particular, it only depends on the set $\Theta_K$) such that 
$$|\theta_K(k_L k(u_Q)) j(u_Q)-c\theta_K(k_L)|=|\theta_K(k_L k(u_Q)) j(u_Q)-\theta_K(k_L)j(1)|\leq C\cdot e^{-\delta'\cdot \sigma_{G/Z_G}(a)},$$
$$|\theta_K(k_L k(u_Q)) j(u_Q)-c\theta_K(k_L)|\leq C\cdot ||\theta_K||$$
for all $\theta_K\in \Theta_K$, $a\in A^{Q,+}(\delta)$, $u_Q\in a\bar{U}_Q[<\epsilon \sigma_{G/Z_G}(a)]a^{-1}$, $k_L\in K_L$. This implies that 
$$|\theta_K(k_L k(u_Q)) j(u_Q)-c\theta_K(k_L)| \leq C\cdot e^{-\frac{\delta'}{2}\cdot \sigma_{G/Z_G}(a)}\cdot \sqrt{||\theta_K||}$$
for all $\theta_K\in \Theta_K$, $a\in A^{Q,+}(\delta)$, $u_Q\in a\bar{U}_Q[<\epsilon \sigma_{G/Z_G}(a)]a^{-1}$, $k_L\in K_L$. Hence in order to prove (2), it is enough to prove the following statement:

\begin{itemize}
\item[(3)] For $\epsilon>0$ small, we have
$$|f(a^{-1}k_L k(u_Q)a)-f(a^{-1}k^Qa)|< \nu_d(f)e^{-\delta'\cdot \sigma_{G/Z_G}(a)}\Xi^G(a^{-1}k^Qa)\sigma_{G}(a^{-1}k^Qa)^{-d'}$$
for all $a\in A^{Q,+}(\delta)$, $u_Q\in a\bar{U}_Q[<\epsilon \sigma_{G/Z_G}(a)]a^{-1}$, $k_L\in K_L$, $f\in \CC(G)$, and for some norm $\nu_d$ on $\CC(G)$.
\end{itemize}
This follows from the same argument as (8.1.17) of \cite{B15}. Now we have finished the proof of the proposition.
\end{proof}

\begin{rmk}
From the proof of the proposition, we know that we can replace $||\theta_K||$ by $||\theta_{K,f}||$ where
$$||\theta_{K,f}||=\max_{k\in K\cap Supp(\theta_f)} |\theta_K(k)|.$$
We can also replace the space $\CC_{scusp}(G)$ by $\CC_{scusp}(G/A_{G}^{\circ})$-- the Mellin transform of the space $\CC_{scusp}(G)$ with respect to the trivial character of $A_{G}^{\circ}$.
\end{rmk}

\begin{defn}
For $f\in \CC_{scusp}(G/A_{G}^{\circ})$, define
$$I(f,\theta_K)=\int_{KA_{G}^{\circ}\back G} I(f,x,\theta_K)dx.$$
By Proposition \ref{convergence} above and Proposition 1.5.1(v) of \cite{B15}, we know the integral is absolutely convergent.
\end{defn}

\begin{rmk}
Note that the double integral 
$$\int_{KA_{G}^{\circ}\back G}\int_{K} f(x^{-1}kx)\theta_K(k) dk dx$$
is not absolutely convergent in general. 
\end{rmk}

\subsection{The spectral side}\label{sec spectral definition}
Let $\omega$ be a finite dimensional representation of $K$ and $\theta_\omega(k)=\tr(\omega(k)),\;k\in K$ be its character. For $f\in \CC_{scusp}(G/A_{G}^{\circ})$, we define the spectral side of the trace formula to be 
$$I_{spec}(f,\omega)=\int_{\CX(G/A_{G}^{\circ})} D(\pi)\theta_f(\pi) m(\bar{\pi},\omega^{\vee})d\pi.$$
Here $\CX(G/A_{G}^{\circ})$ (resp. $\CX_{ell}(G/A_{G}^{\circ})$) is the set of virtual tempered representations (resp. elliptic representations) of $G$ whose central character is trivial on $A_{G}^{\circ}$ defined in Section 2.7 of \cite{B15}. The number $D(\pi)$ and the measure $d\pi$ were also defined in Section 2.7 of \cite{B15}. As in Section 5.4 of loc. cit., for $\pi\in \CX(G/A_{G}^{\circ})$, we can define a map
$$f\in \CC_{scusp}(G/A_{G}^{\circ})\mapsto \theta_f(\pi)\in \BC$$
via the weighted character (this map was denoted by $f\mapsto \hat{\theta}_f(\pi)$ in loc. cit.). 

\subsection{The geometric side}\label{sec geometric definition}
In this section we will define the geometric side of the trace formula and the geometric multiplicity. These have already been defined in \cite{Wan19} when $K$ is connected (i.e. $K=K^\circ$). We just needs to slightly extend the definitions to the non-connected case. We also need to regularize the integral in the geometric multiplicity.

\begin{defn}\label{support of geometric multiplicity}(the support of geometric multiplicity) Let $\CS(G,K)$ be the set of $K$-conjugacy classes $x\in K$  such that the pair $(G_x,K_x)$ is a minimal spherical pair. We refer the reader to Section 2.6 of \cite{Wan19} for the definition of minimal spherical pair. The set $\CS(G,K)$ is the support of the geometric multiplicity.
\end{defn}

\begin{rmk}
\begin{enumerate}
\item The pair $(G_x,K_x)$ is a minimal spherical pair if and only if $G_x$ is split modulo the center.
\item This definition is just Definition 4.1 of \cite{Wan19}. Here since $K$ is compact, the elliptic condition in loc. cit. is automatic. Also the quasi-split condition in loc. cit. is a direct consequence of the condition that $(G_x,K_x)$ is a minimal spherical pair.
\end{enumerate}

\end{rmk}

In order to define a measure on $\CS(G,K)$, we will give an equivalent definition of $\CS(G,K)$. The next definition is an analog of Definition 4.3 of \cite{Wan19}.

\begin{defn}\label{defn support}
Let $\CT(G,K)$ be the set of all the closed (not necessarily connected) abelian subgroups $T$ of $K$ (up to $K$-conjugation) satisfying the following three conditions.
\begin{enumerate}
\item The pair $(G_T,K_T)$ is a minimal spherical pair.
\item We have $T=Z_{Z_G(T)}\cap K$ where $Z_{Z_G(T)}$ is the center of $Z_G(T)$. In particular, we have $Z_{G,K}=Z_G\cap K\subset T$. 
\item There exists $t\in T$ such that $(G_t,K_t)=(G_T,K_T)$.
\end{enumerate}
Let $\CT(G,K)^\circ=\{T\in \CT(G,K)|\; T=T^\circ Z_{G,K}\}$ where $T^\circ$ is the neutral component of $T$ which is a subtorus of $K^\circ$.
\end{defn}

\noindent
For $T\in \CT(G,K)$, there exists a nonempty  subset $C(T,K)$ of the component group $T/T^\circ$ satisfying the following two conditions:
\begin{itemize}
\item For $\gamma \in C(T,K)$, $(G_t,K_t)=(G_T,K_T)$ for almost all $t\in \gamma T^\circ$.
\item For $\gamma \in T/T^\circ-C(T,K)$, $(G_t,K_t)\neq (G_T,K_T)$ for all $t\in \gamma T^\circ$.
\end{itemize}

\begin{defn}
For $T\in \CT(G,K)$, let $T_K=\cup_{\gamma \in C(T,K)} \gamma T^\circ\subset T\subset K$ and $W(K,T)=N_K(T)/Z_K(T)$ where $N_K(T)$ is the normalizer of $T$ in $K$. Let $T_K'$ be the Zariski open subset of $T_K$ consisting of those elements $t\in T_K$ such that $(G_t,K_t)=(G_T,K_T)$.
\end{defn}

\begin{rmk}
For $T\in \CT(G,K)^\circ$, $(G_t,K_t)=(G_T,K_T)$ for almost all $t\in T$ and we have $T_K=T$.
\end{rmk}

\begin{lem}
The support of the geometric multiplicity $\CS(G,K)$ is equal to the set $\cup_{T\in \CT(G,K)} T_K'//W(K,T)$.
\end{lem}

\begin{proof}
From the definition it is clear that $T_K'//W(K,T)$ belongs to the support $\CS(G,K)$ for all $T\in \CT(G,K)$. For the other direction, given $t\in \CS(G,K)$, let $T=Z_{Z_G(t)}\cap K$. Then it is easy to see that $T\in \CT(G,K)$ and $t\in T_K'//W(K,T)$. This proves the lemma.
\end{proof}

The lemma above gives us a natural measure on the set $\CS(G,K)$. More specifically, since $T_K$ is a finite union of translations of the subtori $T^\circ$, the Haar measure on $T^\circ$ induces a measure on $T_K$ (we choose the Haar meausre on $T^\circ$ so that the total volume is equal to 1) such that $T_K-T_K'$ has measure zero (because $T_K'$ is a Zariski open subset of $T_K$). This gives us a measure on $T_K'$ and hence a measure on the support $\CS(G,K)$.

\begin{lem}\label{homogeneous degree}
For $T\in \CT(G,K)$, we have $\dim(K_T)=\frac{\delta(G_T)}{2}+\dim(T)$ and
$$\dim(K)-\dim(K_T)-\frac{\dim(G)-\dim(G_T)}{2}+\frac{\delta(G)}{2}+\dim(T)=\dim(K)$$
\end{lem}

\begin{proof}
By the definition of $\CT(G,K)$, we know that (here $B_T$ is a Borel subgroup of $G_T$) 
$$\dim(K_T)=\dim(G_T)-\dim(B_T)+\dim(K_T\cap Z_{G_T})=\dim(G_T)-\dim(B_T)+\dim(K\cap Z_{G_T})$$
$$=\dim(G_T)-\dim(B_T)+\dim(T)=\frac{\delta(G_T)}{2}+\dim(T).$$
For the second equation, by the first equation, we have 
$$\dim(K)-\dim(K_T)-\frac{\dim(G)-\dim(G_T)}{2}+\frac{\delta(G)}{2}+\dim(T)$$
$$=\dim(K)-\frac{\delta(G_T)}{2}-\frac{\dim(G)-\dim(G_T)}{2}+\frac{\delta(G)}{2}=\dim(K)-\frac{\dim(G)-rank(G_T)}{2}+\frac{\delta(G)}{2}=\dim(K).$$
This proves the lemma.
\end{proof}

\begin{defn}\label{defn geom multiplicity group}
For $\theta\in QC(G),\;\theta_K\in C^\infty(K)^K$ and $s\in \BC$ ($C^\infty(K)^K$ is the space of smooth $K$-invariant functions on $K$), define
$$m_{geom,G,K,s}(\theta,\theta_K):=\int_{\CS(G,K)} \frac{1}{c(G_t,K\cap G_t,\BR)\cdot |Z_K(t):K\cap G_t|}D^{G}(t)^{1/2}\Delta(t)^{s-1/2}\theta_{K}(t) c_\theta(t)dt$$
where the constant $c(G_t,K\cap G_t,\BR)$ is the number of connected components of $B_t\cap (K\cap G_t)$ ($B_t$ is any Borel subgroup of $G_t$) defined in Section 5 of \cite{Wan19} and $\Delta(t)=D^G(t)D^K(t)^{-2}$.
\end{defn}

We also need to define the Lie algebra analogue of the geometric multiplicity. Let $\CS_{Lie}(G,K)$ be the union of $\Ft$ for $T\in \CT(G,K)^{\circ}$ where $\Ft$ is the Lie algebra of $T^\circ$. The measure on $T^\circ//W(K,T)$ induces a measure on $\CS_{Lie}(G,K)$.

\begin{defn}
For $\theta\in QC_c(\Fg),\;\theta_K\in C^\infty(\Fk)^K$ and $s\in \BC$ ($C^\infty(\Fk)^K$ is the space of smooth $K$-invariant functions on $K$), define
$$m_{geom,G,K,s,Lie}(\theta,\theta_K):=\int_{\CS_{Lie}(G,K)} \frac{1}{c(G_T,K\cap G_T,\BR)}D^{G}(T)^{1/2}\Delta(T)^{s-1/2}\theta_{K}(T) c_\theta(T)dT$$
where the constant $c(G_T,K\cap G_T,\BR)$ is the number of connected components of $B_T\cap (K\cap G_T)$ ($B_T$ is any Borel subgroup of $G_T$) and $\Delta(T)=D^G(T)D^{K}(T)^{-2}$.
\end{defn}

\begin{prop}\label{prop geometric side convergence}
The integral defining $m_{geom,G,K,s}(\theta,\theta_K)$ is absolutely convergent when $\frac{1}{2}>Re(s)>0$. Moreover, the limit
$$\lim_{s\rightarrow 0^+} m_{geom,G,K,s}(\theta,\theta_K)$$
exists and it defines a continuous linear form on $QC(G)\times C^\infty(K)^K$.
\end{prop}

\begin{rmk}
The topology on $QC(G)$ was defined in Section 4.4 of \cite{B15}. The topology on $C^\infty(K)^K$ is induced from the usual locally convex topology on $C^\infty(K)$.
\end{rmk}

\begin{proof}
To simplify the notation, we will assume that the center $Z_G$ is trivial. In order to prove this proposition, we first need a lemma.

\begin{lem}\label{lemma geometric side convergence}
The integral $\int_{\CS(G,K)} \Delta(t)^{s-1/2} dt$ is absolutely convergent when $\frac{1}{2}>Re(s)>0$. Moreover, there exists $d>0$ such that 
$$\lim_{s\rightarrow 0^+} s^d\int_{\CS(G,K)} \Delta(t)^{s-1/2} dt=0.$$
\end{lem}

\begin{proof}
This is a technical lemma and we will postpone the proof to the appendix.
\end{proof}

Using the above lemma together with the fact that the function $D^G(t)^{1/2}c_\theta(t)$ is locally bounded (Proposition 4.5.1 of \cite{B15}), we know that the integral defining $m_{geom,G,K,s}(\theta,\theta_K)$ is absolutely convergent when $\frac{1}{2}>Re(s)>0$. It remains to show that the limit  $\lim_{s\rightarrow 0^+} m_{geom,G,K,s}(\theta,\theta_K)$ exists (once we have proved the limit exists, by the uniform boundedness principal principal in Appendix A.1 of \cite{B15} we know that it defines a continuous linear form). We follow the same argument as in Proposition 11.2.1 of \cite{B15}.

For $x\in K$, $K\cap G_x$ is a maximal compact subgroup of $G_x$. For a quasi-character $\theta_x$ on $G_x$, a smooth function $\theta_{K,x}$ on $Z_K(x)\cap G_x$ that is invariant under conjugation, and $s\in \BC$, let $m_{geom,G,K,x,s}(\theta_x,\theta_{K,x})$ be the analogue of $m_{geom,G,K,s}(\theta)$ for the model $(G_x,K\cap G_x)$. In particular, we know that it is well defined when $\frac{1}{2}>Re(s)>0$. Let $\Omega_x\subset G_x$ be a G-good open neighborhood of $x$ (we refer the reader to Section 3.2 of \cite{B15} for the definition of G-good open neighborhood) and set $\Omega=\Omega_{x}^{G}$. If $\Omega_x$ is sufficiently small, by the definition of the set $\CS(G,K)$, we have $\CS(G,K)\cap \Omega=\CS(G_x,K\cap G_x)\cap \Omega_x$. Combining with Proposition 4.5.1.1(iv) of \cite{B15}, we have the following claim which is an analogue of (11.2.11) of \cite{B15}.

\begin{enumerate}
\item[\textbf{Claim 1.}] With the notation above, we have the equality
$$m_{geom,G,K,s}(\theta,\theta_K)=|Z_K(x)/ K\cap G_x|^{-1}\cdot m_{geom,G,K,x,s}((\eta_{x,G})^{s-1/2}\theta_{x,\Omega_x},\eta_{x,K}^{1-2s}\theta_K|_{Z_K(x)\cap G_x})$$
for all $\theta\in QC_c(\Omega)$ and $s\in \BC$ with $\frac{1}{2}>Re(s)>0$. Here $\theta_{x,\Omega_x}$ is the semisimple descent of $\theta$ (which is a quasi-character on $G_x$) defined in Section 3.2 of \cite{B15}, $\eta_{x,G}(y)=D^G(y)D^{G_x}(y)^{-1}$ and $\eta_{x,K}(y)=D^K(y)D^{K_x}(y)^{-1}$.
\end{enumerate}

Then by the same argument as in (11.2.12) of \cite{B15}, the above claim implies the following claim. 

\begin{enumerate}
\item[\textbf{Claim 2.}] The limit $\lim_{s\rightarrow 0^+} m_{geom,G,K,s}(\theta,\theta_K)$ exists when $1\notin Supp(\theta)$ or $1\notin Supp(\theta_K)$.
\end{enumerate}

Let $\omega\subset \Fg$ be a G-excellent open neighborhood of 0 (we refer the reader to Section 3.3 of \cite{B15} for the definition of G-excellent open neighborhood) and set $\Omega=\exp(\omega)$. We only need to show that the limit $\lim_{s\rightarrow 0^+} m_{geom,G,K,s}(\theta,\theta_K)$ exists for all $\theta\in QC_c(\Omega)$.

In order to prove this, we study the Lie algebra analogue of $m_{geom,G,K,s}(\theta,\theta_K)$. Like the group case, we know that the integral defining $m_{geom,G,K,s,Lie}(\theta,\theta_K)$ is absolutely convergent for $\frac{1}{2}>Re(s)>0$. The following claim is a direct consequence of the definition of $m_{geom,G,K,s}(\theta,\theta_K)$ and $m_{geom,G,K,s,Lie}(\theta,\theta_K)$.

\begin{enumerate}
\item[\textbf{Claim 3.}] For $\theta\in QC_c(\Omega)$ and $\theta_K\in C^\infty(K)^K$ with $Supp(\theta_K)\subset \Omega$, we have
$$m_{geom,G,K,s}(\theta,\theta_K)=m_{geom,G,K,s,Lie}((j_{G})^{s-1/2}\theta_{\omega},(j_K)^{1/2-2s}\theta_{K,\omega})$$
for all $s\in \BC$ with $\frac{1}{2}>Re(s)>0$. Here $\theta_\omega$ (resp. $\theta_{K,\omega}$) is the descent of $\theta$ (resp. $\theta_K$) to the Lie algebra (defined in Section 3.3 of \cite{B15})  whose support is contained in $\omega$ (resp. $\Fk\cap \omega$) and $j_{G}(X)=D^G(\exp(X))D^G(X)^{-1}$ (resp. $j_{K}(X)=D^K(\exp(X))D^K(X)^{-1}$).
\end{enumerate}

Combining Claim 2 and Claim 3, we know that the limit $\lim_{s\rightarrow 0^+} m_{geom,G,K,s,Lie}(\theta,\theta_K)$ exists for all $\theta\in QC_c(\omega)$ and $\theta_K\in C^\infty(\Fk)^K$ with $0\notin Supp(\theta)$ or $0\notin Supp(\theta_K)$. Moreover, in order to prove the proposition, it is enough to show that the limit $\lim_{s\rightarrow 0^+} m_{geom,G,K,s,Lie}(\theta,\theta_K)$ exists for all $\theta\in QC_c(\Fg)$ and $\theta_K\in C^\infty(\Fk)^K$.

By Lemma \ref{homogeneous degree} and Lemma \ref{lemma geometric side convergence}, together with the same argument as in the proof of (11.2.18) and (11.2.19) of \cite{B15}, we have the following two claims.

\begin{enumerate}
\item[\textbf{Claim 4.}] There exists $d>0$ such that for all $\theta\in QC_c(\Fg)$, $\theta_K\in C^\infty(\Fk)^K$ and $\lambda\in \BR^\times$, we have
$$\lim_{s\rightarrow 0^+} \sum_{i=0}^{d} (-1)^{d-i} \frac{d!}{i!(d-i)!}m_{geom,G,K,s,Lie}((M_{\lambda,\dim(K)})^i\theta,(M_{\lambda,0})^i\theta_{K})=0.$$
\item[\textbf{Claim 5.}] The limit $\lim_{s\rightarrow 0^+} m_{geom,G,K,s,Lie}(\theta,\theta_K)$ exists for all $\theta\in QC_c(\Fg)$ and $\theta_K\in C^{\infty}(\Fk)^K$ with $0\notin Supp(\theta)$ or $0\notin Supp(\theta_K)$. 
\end{enumerate}

We first consider the case when $\theta_K$ is the constant function. In this case, Claim 4 above becomes
\begin{itemize}
\item[\textbf{Claim 6.}] There exists $d>0$ such that for all $\theta\in QC_c(\Fg)$ and $\lambda\in \BR^\times$, we have
$$\lim_{s\rightarrow 0^+} m_{geom,G,K,s,Lie}((M_{\lambda,\dim(K)}-1)^d\theta,\theta_{K})=0.$$
\end{itemize}

If $G$ is not split, we have $\dim(K)>\delta(G)/2$. By Claim 5, Claim 6 and Proposition \ref{lemma quasi-character homogeneous}, we know that the limit $\lim_{s\rightarrow 0^+} m_{geom,G,K,s,Lie}(\theta,\theta_K)$ exists for all $\theta\in QC_c(\Fg)$. 

If $G$ is split, by Claim 5, Claim 6 and Proposition \ref{lemma quasi-character homogeneous}, we know that the limit $\lim_{s\rightarrow 0^+} m_{geom,G,K,s,Lie}(\theta,\theta_K)$ exists for all $\theta\in QC_c(\Fg)$ such that $c_{\theta,\CO}=0$ for all $\CO\in Nil_{reg}(\Fg)$. Hence it is enough to show that the limit exists when $\theta\in Span\{\hat{j}(\CO,\cdot)|\; \CO\in Nil_{reg}(\Fg)\}$ in a neighborhood of $0$.

For $\CO_1,\CO_2\in Nil_{reg}(\Fg)$, if $\theta$ is equal to $\hat{j}(\CO_1,\cdot)-\hat{j}(\CO_2,\cdot)$ in a neighborhood of $0$, let $g\in Res_{\BC/\BR}G$ be an element as in Lemma \ref{lemma nilpotent orbit} and we can write $\theta$ as $\theta_1-\theta_2$ where $\theta_i$ is equal to $\hat{j}(\CO_i,\cdot)$ in a neighborhood of $0$. We know that $0\notin Supp({}^g\theta_1-\theta_2)$ where ${}^g\theta_1(X)=\theta_1(g^{-1}Xg)$. This implies that   $$\lim_{s\rightarrow 0^+}m_{geom,G,K,s,Lie}(\theta,\theta_K)=\lim_{s\rightarrow 0^+}m_{geom,G,K,s,Lie}(\theta_1-\theta_2,\theta_K)$$
$$=\lim_{s\rightarrow 0^+}m_{geom,G,K,s,Lie}({}^g\theta_1-\theta_2,\theta_K)$$
exists. As a result, it is enough to show that the limit exists when $\theta=\sum_{\CO\in Nil_{reg}(\Fg)} \hat{j}(\CO,\cdot)$ in a neighborhood of $0$. But this follows from the following two facts  
\begin{itemize}
\item $\theta=\sum_{\CO\in Nil_{reg}(\Fg)} \hat{j}(\CO,\cdot)$ is parabolically induced from a maximal torus of a Borel subgroup of $G$ (Section 3.4 of \cite{B15});
\item the intersection of a Borel subgroup of $G$ with $K$ is finite.
\end{itemize}

Now we consider the case when $\theta_K$ is not the constant function. Since we have already proved the limit exists when $\theta_K$ is a constant function. We may assume that $\theta_K(0)=0$. Since $\theta\in QC_c(\Fg)$, we may also assume that $\theta_K$ is compactly supported.

Consider the space $V=QC(\Fg)\hat{\otimes} C^\infty(\Fk)^K$ realized as functions on $\Fg_{reg}\times \Fk$. Let $V_c\subset V$ be the subspace consisting of functions whose support is compact modulo conjugation. For $\Theta\in QC(\Fg)\hat{\otimes} C^\infty(\Fk)^K$, $\Theta(\cdot,X)$ is a quasi-character on $\Fg$ for all $X\in \Fk$. Hence we can still define the regular germ $c_{\Theta}(Y,X):=c_{\Theta(\cdot,X)}(Y)$. For $Re(s)>0$, we can also define $m_{geom,G,K,s,Lie}(\Theta)$ for $\Theta\in V_c$. Claim 4 implies that 
$$\lim_{s\rightarrow 0^+} m_{geom,G,K,s,Lie}((M_{\lambda,\dim(K)}-1)^d \Theta)=0$$ 
for all $\Theta\in V_c$. Here $(M_{\lambda,\dim(K)})\Theta(X,Y)=|\lambda|^{-\dim(K)}\Theta(\lambda^{-1}X,\lambda^{-1}Y)$. Claim 5 implies that the limit $\lim_{s\rightarrow 0^+}m_{geom,G,K,s,Lie}(\Theta)$ exists when $\Theta\in V_c$ and $0\notin Supp(\Theta)$.

Let $\Theta=\theta\times \theta_K\in V_c$. We need a lemma.

\begin{lem}\label{homogeneous projective}
For $\lambda\in \BR^\times$ with $\lambda\neq \pm 1$, we can write $\Theta$ as $(M_{\lambda,\dim(K)}-1)^d \Theta_1+\Theta_2$ for some $\Theta_1,\Theta_2\in V_c$ with $0\notin Supp(\Theta_2)$.
\end{lem}

\begin{proof}
The proof is very similar to Proposition 4.6.1(i) of \cite{B15}. We may assume that $|\lambda|>1$. It is enough to show that the sequence
$$\sum_{i=0}^{\infty} \frac{(i+d-1)!}{i!}(M_{\lambda,\dim(K)})^i \Theta$$
converges in $V$. Let $L$ be a compact subset of $\Fg\times \Fk$. Let $I(\Fg\oplus \Fk)$ be the space of $G\times K$-invariant differential operators on $\Fg\oplus \Fk$. We only need to show that 
$$\sum_{i=0}^{\infty} \frac{(i+d-1)!}{i!}q_{L,u}((M_{\lambda,\dim(K)})^i \Theta)$$
converges for all $u\in I(\Fg\oplus \Fk)$ where 
$$q_{L,u}\Theta_0=\sup_{(X,Y)\in L\cap (\Fg_{reg}\times \Fk)} D^G(X)^{1/2} |\partial(u)\Theta_0(X,Y)|,\;\Theta_0\in V.$$
We may assume that $u$ is homogeneous. It is enough to show that
\begin{equation}\label{homogeneous projective 1}
q_{L,u}((M_{\lambda,\dim(K)})^i \Theta)\ll |\lambda|^{-i}
\end{equation}
for all $i\geq 0$.

If $\deg(u)>0$, by enlarging $L$ we can assume that $\lambda^{-1}L\subset L$. Then for $i\geq 1$, we have
$$q_{L,u}((M_{\lambda,\dim(K)})^i \Theta)=|\lambda|^{-i\deg(u)-i(\dim(K)-\delta(G)/2)} \sup_{(X,Y)\in L\cap (\Fg_{reg}\times \Fk)} D^G(\lambda^{-i}X)^{1/2} |\partial(u)\Theta(\lambda^{-i}X,\lambda^{-i}Y)|$$
$$\leq |\lambda|^{-i\deg(u)}\sup_{(X,Y)\in \lambda^{-i}L\cap (\Fg_{reg}\times \Fk)} D^G(X)^{1/2} |\partial(u)\Theta(X,Y)|\ll |\lambda|^{-i\deg(u)} q_{L,u}(\Theta).$$
This proves \eqref{homogeneous projective 1}. If $u=1$, since $\theta_K(0)=0$, we have 
$$D^G(X)^{1/2}\Theta(X,Y)=O(|Y|)$$
for $(X,Y)\in \Fg_{reg}\times \Fk$ close to 0. Hence we have
$$q_{L,1}((M_{\lambda,\dim(K)})^i \Theta)=|\lambda|^{-i(\dim(K)-\delta(G)/2)}\sup_{(X,Y)\in \lambda^{-i}L\cap (\Fg_{reg}\times \Fk)} D^G(X)^{1/2} |\Theta(X,Y)|\ll |\lambda|^{-i}$$
for all $i\geq 0$. This finishes the proof of the lemma.
\end{proof}

By the lemma above, we have
$$\lim_{s\rightarrow 0^+} m_{geom,G,K,s,Lie}(\theta,\theta_K) = \lim_{s\rightarrow 0^+} m_{geom,G,K,s,Lie}(\Theta)$$
$$=\lim_{s\rightarrow 0^+} m_{geom,G,K,s,Lie}((M_{\lambda,\dim(K)}-1)^d \Theta_1)+m_{geom,G,K,s,Lie}(\Theta_2)=\lim_{s\rightarrow 0^+} m_{geom,G,K,s,Lie}(\Theta_2).$$
In particular, we know that $\lim_{s\rightarrow 0^+} m_{geom,G,K,s,Lie}(\theta,\theta_K)$ exists. This finishes the proof of the proposition.
\end{proof}

\begin{defn}
For $\theta\in QC(G)$ and $\theta_K\in C^{\infty}(K)^K$, define 
$$m_{geom,G,K}(\theta,\theta_K)=\lim_{s\rightarrow 0^+} m_{geom,G,K,s}(\theta,\theta_K).$$
For $\theta\in QC_c(\Fg)$ and $\theta_K\in C^{\infty}(\Fk)^K$, define
$$m_{geom,G,K,Lie}(\theta,\theta_K)=\lim_{s\rightarrow 0^+} m_{geom,G,K,s,Lie}(\theta,\theta_K).$$
\end{defn}

The next Corollary is a direct consequence of the proof of the proposition above.

\begin{cor}\label{geometric multiplicity descent}
\begin{enumerate}
\item Let $x\in K$ and $\Omega_x\in G_x$ be a $G$-good open neighborhood of $x$ and $\Omega=\Omega_{x}^{G}$. Then if $\Omega_x$ is sufficiently small, we have
$$m_{geom,G,K}(\theta,\theta_K)=|Z_K(x)/ K\cap G_x|^{-1}\cdot m_{geom,G,K,x}((\eta_{x,G})^{-1/2}\theta_{x,\Omega_x},\eta_{x,K}\theta_K|_{K\cap G_x})$$
for all $\theta\in QC_c(\Omega)$. Here $m_{geom,G,K,x}(\cdot, \cdot)$ is the analogue of $m_{geom,G,K}(\cdot,\cdot)$ for the model $(G_x,K\cap G_x)$.
\item For $\theta\in QC_c(\Fg)$ and $\theta_K\in C^{\infty}(\Fk)^K$, we have
$$m_{geom,G,K,Lie}(\theta,\theta_K)=|\lambda|^{-\dim(K)}m_{geom,G,K,Lie}(\theta_\lambda,\theta_{K,\lambda}).$$
\item Let $\omega\subset \Fg$ be a G-excellent open neighborhood of $0$ and let $\Omega=\exp(\omega)$. Then 
$$m_{geom,G,K}(\theta,\theta_K)=m_{geom,G,K,Lie}((j_{G})^{-1/2}\theta_{\omega},j_{K}^{1/2}\theta_{K,\omega})$$
for all $\theta\in QC_c(\Omega)$ and $\theta_K\in C^{\infty}(K)^K$ with $Supp(\theta_K)\subset \Omega\cap K$.
\end{enumerate}
\end{cor}

\begin{defn}
For $f\in \CC_{scusp}(G/A_{G}^{\circ})$, define
$$I_{geom}(f,\theta_K)=m_{geom,G,K}(\theta_f,\theta_K).$$
If $\pi$ is a finite length smooth representation of $G$, and $\omega$ is a finitely dimensional representation of $K$, define the geometric multiplicity
$$m_{geom}(\pi,\omega)=m_{geom,G,K}(\theta_\pi,\theta_{\omega^{\vee}})=m_{geom,G,K}(\theta_\pi,\overline{\theta_{\omega}})$$
where $\omega^{\vee}$ is the dual representation of $\omega$.
\end{defn}

\subsection{The trace formula and the multiplicity formula}\label{sec trace formula imply multiplicity formula}

\begin{thm}\label{thm trace formula}
Let $\theta_K$ be a smooth function on $K$ that is invariant under conjugation. For $f\in \CC_{scusp}(G/A_{G}^{\circ})$, we have
$$I(f,\theta_K)=I_{geom}(f,\theta_K).$$
If $\theta_K=\theta_{\omega}$ where $\omega$ is a finite dimensional representation of $K$, then 
$$I(f,\theta_K)=I_{spec}(f,\omega).$$
\end{thm}

\begin{thm}\label{thm multiplicity formula}
Let $\omega$ be a finitely dimensional representation of $K$. The multiplicity formula
$$m(\pi,\omega)=m_{geom}(\pi,\omega)$$
holds for all smooth finite length representations $\pi$ of $G$.
\end{thm}

\begin{rmk}
It is enough to prove the multiplicity formula when $\pi$ and $\omega$ are irreducible. Moreover, up to twist $\pi$ by some character, we only need to prove the multiplicity formula when the central character of $\pi$ is trivial on $A_{G}^{\circ}$.
\end{rmk}

\subsection{Trace formula implies multiplicity formula}
In this subsection, we will assume that the trace formula in Theorem \ref{thm trace formula} holds, the goal is to prove the multiplicity formula in Theorem \ref{thm multiplicity formula}. The key ingredient of the proof is to show that both the multiplicity and the geometric multiplicity behave nicely under parabolic induction.

To be specific, let $P=MN$ be a parabolic subgroup of $G$. Up to conjugating $M$ we may assume that $P\cap K\subset M$. Let $K_M=P\cap K=M\cap K$. It is a maximal compact subgroup of $M$. Let $\tau$ be a smooth finite length representation of $M$ and $\pi=I_{P}^{G}(\tau)$. Let $\omega$ be a finitely dimensional representation of $K$ and $\omega_M=\omega|_{K_M}$. We need to prove the following proposition.

\begin{prop}
With the notation above, we have
$$m(\pi,\omega)=m(\tau,\omega_M),\;m_{geom}(\pi,\omega)=m_{geom}(\tau,\omega_M).$$
\end{prop}

\begin{proof}
The first identity $m(\pi,\omega)=m(\tau,\omega_M)$ follows from the Iwasawa decomposition $G=PK$ and the reciprocity law. Now we prove the second identity $m_{geom}(\pi,\omega)=m_{geom}(\tau,\omega_M)$.

We have a natural map from $\iota_{M,G}:M_{ss}/conj$ to $G_{ss}/conj$. For $t_M\in M_{ss}/conj$, it is clear from the definition that 
\begin{equation}\label{parabolic 1}
t_M\in \CS(M,K_M) \iff \iota_{M,G}(t_M)\in  \CS(G,K).
\end{equation}

For $t_M\in \CS(M,K_M)$, let $t=\iota_{M,G}(t_M)\in \CS(G,K)$. We have
$$D^{K}(t)D^{K_M}(t_M)^{-1}=|\det(1-Ad(x))_{|\Fn/\Fn_x}|.$$
This implies that
\begin{equation}\label{parabolic 2} D^{K}(t)D^{G}(t)^{-1/2}=D^{K_M}(t_M)D^M(t_M)^{-1/2}.
\end{equation}
Meanthile, We fix a representative of the conjugacy class $t_M$ and denote it by $t_0$. Then $t_0$ is also a representative of the conjugacy class $t$. We want to show that
\begin{equation}\label{parabolic 3}
\frac{|Z_G(t_0):G_{t_0}|}{|Z_K(t_0):(K\cap G_{t_0})|\cdot c(G_{t_0},K\cap G_{t_0},\BR)}=\frac{|Z_M(t_0):M_{t_0}|}{|Z_{K_M}(t_0):(K_M\cap M_{t_0})|\cdot c(M_{t_0},K_M\cap G_{t_0},\BR)}.
\end{equation}
Since $K\cap G_{t_0}$ (resp. $K_M\cap M_{t_0}$) a maximal compact subgroup of $G_{t_0}$ (resp. $M_{t_0}$) and $Z_K(t_0)$ (resp. $Z_{K_M}(t_0)$) is a maximal compact subgroup of $Z_G(t_0)$ (resp. $Z_M(t_0)$), we have
$$\frac{|Z_G(t_0):G_{t_0}|}{|Z_K(t_0):(K\cap G_{t_0})|}=\frac{|Z_M(t_0):M_{t_0}|}{|Z_{K_M}(t_0):(K_M\cap M_{t_0})|}.$$
Hence it is enough to show that $c(G_{t_0},K\cap G_{t_0},\BR)=c(M_{t_0},K_M\cap M_{t_0},\BR)$. But this just follows from the definition of $c(G,K,\BR)$ and the fact that $M_{t_0}$ is a Levi subgroup of $G_{t_0}$.

Then the identity $m_{geom}(\pi,\omega)=m_{geom}(\tau,\omega_M)$ follows from the definition of the geometric multiplicities, \eqref{parabolic 1}, \eqref{parabolic 2}, \eqref{parabolic 3}, and Proposition \ref{germ parabolic induction}.
\end{proof}

Now we prove the multiplicity formula. By induction, we will assume that the multiplicity formula holds for all proper Levi subgroups of $G$. Combining with the proposition above, we only need to prove the multiplicity formula for elliptic representations. In particular, this proves the multiplicity formula for all representations of $G$ when $G$ does not admit a maximal elliptic torus ($\iff$ there is no elliptic representation of $G$). 

Assume that $G$ admit a maximal elliptic torus. The multiplicity formula will be a direct consequence of the trace formula and the proposition above. The argument is the same as the GGP case (Proposition 11.5.1 of \cite{B15}), we will only give a sketch of the proof. First, as in Proposition 11.5.1(ii) of \cite{B15}, using the trace formula, the proposition above, and the assumption that the multiplicity formula holds for all proper Levi subgroups, we know that for all strongly cuspidal function $f\in \CC_{scusp}(G/A_{G}^{\circ})$, we have
\begin{equation}\label{multiplicity formula equation 1}
\sum_{\pi\in \CX_{ell}(G/A_{G}^{\circ})} D(\pi)\theta_f(\pi)(m(\bar{\pi},\omega^{\vee})-m_{geom}(\bar{\pi},\omega^{\vee}))=0.
\end{equation}
Then by Corollary 5.7.2(iv) of \cite{B15}, for each $\pi\in \CX_{ell}(G/A_{G}^{\circ})$, there exists $f_\pi\in \CC_{scusp}(G/A_{G}^{\circ})$ such that for all $\pi'\in \CX_{ell}(G/A_{G}^{\circ})$, $\theta_f(\pi')\neq 0$ if and only if $\pi=\pi'$. Put $f_\pi$ in the equation \eqref{multiplicity formula equation 1}, we have
$$D(\pi)\theta_{f_\pi}(\pi)(m(\bar{\pi},\omega^{\vee})-m_{geom}(\bar{\pi}),\omega^{\vee})=0,$$
which implies that $m(\bar{\pi},\omega^{\vee})=m_{geom}(\bar{\pi},\omega^{\vee})$. This proves the multiplicity formula.

In the next few sections we will prove the trace formula. By the above discussion, we only need to consider the case when $G$ admit a maximal elliptic torus. By induction, we will assume that the trace formula holds for all groups whose dimension is less than $G$. Also to simplify the notation we will assume that the center of $G$ is trivial for the rest of this paper.

\section{The spectral side}\label{sec spectral side}

\subsection{A reformulation}
Let $G'=G\times K$ and we diagonally embeds $K$ into $G'$. For $f'\in \CC_{scusp}(G')$, define
$$I_{G'}(f',x)=\int_{K} f'(x^{-1}kx)dx,\;I_{G'}(f')=\int_{K\back G'} I_{G'}(f',x),\;I_{G',spec}(f')=\int_{\CX(G')} D(\pi')\theta_{f'}(\pi') m(\bar{\pi}')d\pi'$$
where $m(\pi')=\dim(\Hom_{K}(\pi',1))$. By choosing $f'=f\times \theta_{\omega}$ with $f\in \CC_{scusp}(G)$ we know that in order to prove the spectral expansion, we only need to show that
$$I_{G'}(f')=I_{G',spec}(f')$$
for all $f'\in \CC_{scusp}(G')$.

\subsection{Explicit intertwining operator}

\begin{lem}
Let $\pi'$ be an irreducible tempered representation of $G'$ and $l\in \Hom_K(\pi',1)$ be a $K$-invariant linear form. Then there exists $d>0$ and a continuous semi-norm $\nu_d$ on $\pi'$ such that
$$|l(\pi'(x)e)|\leq \nu_d(e)\Xi^{G'}(x)\sigma_{G'}(x)^d$$
for all $e\in \pi'$ and $x\in G'$.
\end{lem}

\begin{proof}
Let $P_{min}$ be a minimal parabolic subgroup of $G'$ (it is the product of a minimal parabolic subgroup of $G$ with $K$). We only need to prove the inequalities for $x\in A^+$.

Let $X_1,\cdots,X_p$ be a basis of $\Fp_{min}$ and let
$$\Delta_{min}=1-(X_{1}^{2}+\cdots+X_{p}^{2})\in \CU(\Fp_{min}).$$
Let $k$ be an integer greater than $\dim(P_{min})+1$. By elliptic regularity, there exists $\varphi_1\in C_{c}^{k_1}(P_{min})$ and $\varphi_2\in C_{c}^{\infty}(P_{min})$ with $k_1=2k-\dim(P_{min})-1$ such that
$$\pi'(\varphi_1)\pi'(\Delta_{min}^{k})+\pi'(\varphi_2)=Id_{\pi'}$$
for all $\pi'$. Let $\varphi_K$ be a smooth function on $K$ such that $\int_K \varphi_K(k)dk=1$. Then we have
$$l(\pi'(a)e)=l(\pi'(\varphi_1)\pi'(\Delta_{min}^{k})\pi'(a)e)+l(\pi'(\varphi_2)\pi'(a)e)$$
$$=l(\pi'(\varphi_1)\pi'(a)\pi'(a^{-1}\Delta_{min}^{k}a)e)+l(\pi'(\varphi_2)\pi'(a)e)$$
$$=l(\pi'(\varphi_K\ast\varphi_1)\pi'(a)\pi'(a^{-1}\Delta_{min}^{k}a)e)+l(\pi'(\varphi_K\ast \varphi_2)\pi'(a)e).$$
It is clear that the map $a\in A^+\mapsto a^{-1}\Delta_{min}^{k}a\in \CU(\Fg)$ has bounded image. Moreover, we know that the functions $\varphi_K\ast \varphi_i$ ($i=1,2$) belong to $C_{c}^{k_1}(G')$. Hence the lemma follows from the following fact which is (7.3.7) of \cite{B15}.

\begin{itemize}
\item There exists $k_1'\geq 1$ such that for all $\varphi\in C_{c}^{k_1'}(G')$, there exists a continuous semi-norm $\nu_\varphi$ on $\pi'$ such that
$$|l(\pi'(\varphi)\pi'(g)e)|\leq \nu_\varphi(e)\Xi^{G'}(g)$$
for all $e\in \pi'$ and $g\in G'$.
\end{itemize}
\end{proof}

\begin{cor}
For $l\in \Hom_{K}(\pi',1),v\in \pi'$ and $f\in \CC(G')$, the integral
$$\int_{G'} f(g)l(\pi'(g)v)dg$$
is absolutely convergent and is equal to $l(\pi'(f)v)$.
\end{cor}

\begin{proof}
The proof is the same as (7.5.1) and (7.5.4) of \cite{B15}, we will skip it here.
\end{proof}

Fix a $G'$-invariant scalar product $(\cdot,\cdot)$ on $\pi'$. Define a sesquilinear form
$$\CB_{\pi'}:\pi'\times \pi'\rightarrow \BC,\;\CB_{\pi'}(v,v')=\int_K (\pi'(k)v,v')dk.$$
The integral is convergent since $K$ is compact. It factors through $\CB_{\pi'}:\pi_K'\times \pi_K'\rightarrow \BC$.
where $\pi_K'$ is the space of $K$-coinvariants in $\pi'$. The next lemma follows from the fact that $K$ is compact.

\begin{lem}\label{lem linear form nondegenerat}
The sesquilinear form $\CB_{\pi'}:\pi_K'\times \pi_K'\rightarrow \BC$ is non degenerate and the map
$$v\in \pi'\mapsto \CB_{\pi'}(\cdot,v)\in \Hom_K(\pi',1)$$
is surjective
\end{lem}

\begin{rmk}
The above lemma is highly nontrivial for general strongly tempered spherical pairs (although we expect it to be true), especially when the spherical pair is not a Gelfand pair, i.e. when the multiplicity is greater than 1.
\end{rmk}

We also define the linear form $\CL_{\pi'}:\End(\pi')^\infty \rightarrow \BC$ to be (here $\End(\pi')^\infty$ is the space of smooth vectors in $\End(\pi')$ and we refer the reader to Section 2.2 of \cite{B15} for the topology on $\End(\pi')^\infty$)
$$\CL_{\pi'}(T)=\int_K \tr(\pi'(k)T)dk.$$
For $e,e'\in \pi'$, by abusing of notion, we define
$$\CL_{\pi'}(e,e')=\CL_{\pi'}(T_{e,e'})=\int_K(e,\pi'(k)e')dk$$
where $T_{e,e'}\in \End(\pi)^\infty$ is defined to be $T_{e,e'}(e_0)=(e_0,e')e$. The lemma above implies that 
$$\CL_{\pi'}\neq 0\iff m(\pi')\neq 0.$$
Since $\CL_{\pi'}$ is a continuous sesquilinear form on $\pi'$, it defines a continuous linear map 
$$L_{\pi'}:\pi'\rightarrow \overline{\pi'^{-\infty}},\;e\mapsto \CL_{\pi'}(e,\cdot)$$
where $\overline{\pi'^{-\infty}}$ denotes the topological conjugate-dual of $\pi'$ endowed with the strong topology. The operator $L_{\pi'}$ has image in $\overline{\pi'^{-\infty}}^{K}=\Hom_K(\overline{\pi'}, 1)$ which is finite dimensional. For $T\in \End(\pi')^\infty$, it extends uniquely to a continuous operator $T:\overline{\pi'^{-\infty}}\rightarrow \pi'$. This gives two operators
$$TL_{\pi'}:\pi' \rightarrow \pi',\; L_{\pi'}T:\overline{\pi'^{-\infty}}\rightarrow \overline{\pi'^{-\infty}},$$
and we have
$$\tr(TL_{\pi'})=\tr(L_{\pi'}T)=\CL_{\pi'}(T).$$
The next corollary is a direct consequence of Lemma \ref{lem linear form nondegenerat} and it is an analogue of Corollary 7.6.1 of \cite{B15}.

\begin{cor}\label{L_pi cor}
Let $\CK\subset \CX_{temp}(G')$ be a compact subset. Then there exists a section $T\in \CC(\CX_{temp}(G'),\CE(G'))$ such that the restriction of $L_{\pi}T_\pi$ to $\Hom_K(\bar{\pi},1)$ is the identity map for all $\pi\in \CK$. 
\end{cor}

Here we refer the reader to Section 2.6 of \cite{B15} for the definition of $\CX_{temp}(G')$, $\CC(\CX_{temp}(G'),\CE(G'))$ and $C^{\infty}(\CX_{temp}(G'),\CE(G'))$.

\begin{lem}\label{L_pi lem}
\begin{enumerate}
\item The maps 
$$\pi\in \CX_{temp}(G') \mapsto L_\pi \in \Hom(\pi,\overline{\pi^{-\infty}}),\; \pi\in \CX_{temp}(G')\mapsto \CL_\pi\in \End(\pi)^{-\infty}$$
are smooth in the sense of Lemma 7.2.2(1) of \cite{B15}.
\item For all $S,T\in \End(\pi')^{\infty}$ such that the restriction of $L_{\pi'}T$ to $\Hom_K(\bar{\pi}',1)$ is the identity map, we have $SL_{\pi'}\in \End(\pi')^{\infty}$ and $\frac{\CL_{\pi'}(S)\CL_{\pi'}(T)}{m(\pi')}= \CL_{\pi'}(SL_{\pi'}T)$.
\item For $S,T\in \CC(\CX_{temp}(G'),\CE(G'))$, the section $\pi\mapsto S_\pi L_\pi T_\pi\in \End(\pi)^{\infty}$ belongs to $C^{\infty}(\CX_{temp}(G'),\CE(G'))$.
\item Let $f\in \CC(G')$ with compactly supported Fourier transform $\pi\in \CX_{temp}(G')\rightarrow \pi(f)$. We have
$$\int_K f(k)dk=\int_{\CX_{temp}(G')} \CL_\pi (\pi(f))\mu(\pi)d\pi.$$
Here $\mu(\pi)d\pi$ is the Plancherel measure defined in Section 2.6 of \cite{B15}.
\item Let $f$ as in part (iv) and let $f'\in \CC(G')$ such that for all $\pi$ belongs to the support of $f$, the restriction of $L_{\pi}\pi(f')$ to $\Hom_K(\bar{\pi},1)$ is the identity map. Then we have
$$\int_{\CX_{temp}(G')} \frac{\CL_\pi(\pi(f)) \overline{\CL_\pi(\pi(\overline{f'}))}} {m(\pi)} \mu(\pi)d\pi = \int_K\int_K\int_{G'} f(kgk')f'(g)dgdk'dk.$$
\end{enumerate}
\end{lem}

\begin{proof}
The proof of the first part is the same as Lemma 7.2.2(i) of \cite{B15}. For the second part, the proof of $SL_\pi\in \End(\pi)^{\infty}$ follows from the same argument as in Lemma 7.2.2(ii) of \cite{B15}. As for the eqaution, we have $\CL_{\pi'}(SL_{\pi'}T)=\tr(L_{\pi'}SL_{\pi'}T)$. Since the images of $L_{\pi'}S$ and $L_{\pi'}T$ are contained in $\Hom_K(\bar{\pi}',1)$ and the restriction of $L_{\pi'}T$ to $\Hom_K(\bar{\pi}',1)$ is the identity map, we have
$$\CL_{\pi'}(SL_{\pi'}T)=\tr(L_{\pi'}SL_{\pi'}T)=\tr(L_{\pi'}S)=\CL_{\pi'}(S)=\CL_{\pi'}(S)\cdot \frac{\CL_{\pi'}(T)}{m(\pi')}.$$
This proves the second part.

The proof of part (3)-(5) is almost the same as Lemma 7.2.2(iii)-(v) of \cite{B15}, there is only one difference.
\begin{itemize}
\item Compared with Lemma 7.2.2(v) of \cite{B15}, the last part of the lemma has an extra $m(\pi)$ on the left hand side of the equation, this comes from the extra $m(\pi)$ on the second part of the lemma.
\end{itemize}
We will skip the details here.
\end{proof}

\subsection{The proof of the spectral expansion}
In this subsection we are going to show that for all $f\in \CC_{scusp}(G')$, we have
$$I_{G'}(f)=I_{G',spec}(f).$$
It is enough to prove the equation for all $f\in \CC_{scusp}(G')$ whose Fourier transform is compactly supported. For the rest of this subsection we will fix a function $f\in \CC_{scusp}(G')$ whose Fourier transform is compactly supported.

For $f'\in \CC(G')$, define
$$K_{f,f'}^{A}(g_1,g_2)=\int_{G'}f(g_{1}^{-1}gg_2)f'(g)dg,\;g_1,g_2\in G',$$
$$K_{f,f'}^{1}(g,x)=\int_K K_{f,f'}^{A}(g,kx)dk,\;g,x\in G',$$
$$K_{f,f'}^{2}(x,y)=\int_K K_{f,f'}^{1}(kx,y)dk,\;x,y\in G',$$
$$J_{aux}(f,f')=\int_{K\back G'} K_{f,f'}^{2}(x,x)dx.$$
The next lemma is a direct consequence of Theorem 5.5.1(i)-(ii) of \cite{B15} (note that $K$ is compact).

\begin{lem}
\begin{enumerate}
\item Let $\ast\in \{1,2,A\}$.  For all $d\geq 0$, there exists $d'\geq 0$ such that
$$|K_{f,f'}^{\ast}(g_1,g_2)|\ll \Xi^{G'}(g_1)\sigma_{G'}(g_1)^{-d}\Xi^{G'}(g_2)\sigma_{G'}(g_2)^{d'},$$
$$|K_{f,f'}^{\ast}(g_1,g_2)|\ll \Xi^{G'}(g_1)\sigma_{G'}(g_1)^{d'}\Xi^{G'}(g_2)\sigma_{G'}(g_2)^{-d},$$
$$|K_{f,f'}^{\ast}(x,x)|\ll \Xi^{G'}(x)^2 \sigma_{G'}(x)^{-d}$$
for all $g_1,g_2\in G'$ and $x\in G'$.
\item The triple integral defining $J_{aux}(f,f')$
$$\int_{K\back G'}\int_K\int_K K_{f,f'}^{A}(k_1x,k_2x)dk_1dk_2dx$$
is absolutely convergent.
\end{enumerate}
\end{lem}

\begin{prop}\label{J_aux}
We have
$$J_{aux}(f,f')=\int_{\CX(G')} D(\pi) \theta_f(\pi)\overline{\CL_\pi(\pi(\overline{f'}))}d\pi.$$
\end{prop}

\begin{proof}
We have 
$$J_{aux}(f,f')=\int_{K\back G'}\int_K\int_K K_{f,f'}^{A}(k_1x,k_2x)dk_1dk_2dx$$
$$=\int_{G'}\int_K K_{f,f'}^{A}(x,kx)dkdx=\int_K\int_{G'} K_{f,f'}^{A}(x,kx)dxdk$$
$$=\int_K\int_{\CX(G')} D(\pi)\theta_f(\pi)\theta_{\bar{\pi}}(R(k)f')d\pi dk=\int_{\CX(G')}\int_K D(\pi)\theta_f(\pi)\theta_{\bar{\pi}}(R(k)f')dk d\pi$$
$$=\int_{\CX(G')} D(\pi) \theta_f(\pi)\overline{\CL_\pi(\pi(\overline{f'}))}d\pi.$$
where the first equality on the third line follows from Theorem 5.5.1(v) of \cite{B15}.
\end{proof}

Now we are ready to prove the spectral expansion. Recall that we have fixed a function $f\in \CC_{scusp}(G')$ whose Fourier transform is compactly supported. Let $\CK$ be the support of the Fourier transform of $f$. By Lemma \ref{L_pi lem}, we have
$$I_{G'}(f,x)=\int_{\CX_{temp}(G')} \CL_{\pi}(\pi(x)\pi(f)\pi(x^{-1})) \mu(\pi)d\pi.$$
By Corollary \ref{L_pi cor}, there exists a function $f'\in \CC(G')$ such that the restriction of $L_\pi \pi(f')$ to $\Hom_K(\bar{\pi},1)$ is the identity map for all $\pi\in \CK$. Fix such a $f'$. We have 
$$\overline{\CL_\pi(\pi(\overline{f'}))}=m(\pi)=m(\bar{\pi}),\;\pi\in \CK.$$
This implies that 
$$I_{G'}(f,x)=\int_{\CX_{temp}(G')} \frac{\CL_{\pi}(\pi(x)\pi(f)\pi(x^{-1})) \overline{\CL_\pi(\pi(\overline{f'}))}}{m(\pi)} \mu(\pi)d\pi=K_{f,f'}^{2}(x,x)$$
where the last equation follows from Lemma \ref{L_pi lem}. In particular, we have
$$I_{G'}(f)=J_{aux}(f,f')=\int_{\CX(G')} D(\pi) \theta_f(\pi)\overline{\CL_\pi(\pi(\overline{f'}))}d\pi=\int_{\CX(G')} D(\pi) \theta_f(\pi)m(\bar{\pi})d\pi=I_{G',spec}(f).$$
This finishes the proof of the spectral expansion.

\section{The distribution on the Lie algebra}\label{sec Lie algebra}
In this section, we will study the analogue of the distribution $I(f)$ on the Lie algebra when $G$ is split. Let $G$ be a split real reductive group with trivial center, $K=G^\theta$ be a maximal compact subgroup of $G$, $\Fk$ be the Lie algebra of $K$ and $\Fp=\Fk^\perp$ be the orthogonal complement of $\Fk$ in $\Fg$ (i.e. $\Fg=\Fp\oplus \Fk$ is the Cartan decomposition). Let $\Ft\subset \Fp$ be a maximal abelian subspace. Since $G$ is split, we know that $\Ft$ is the Lie algebra of a maximal split torus $T$ of $G$. Let $B=TN$ be a Borel subgroup of $G$, $\bar{B}=T\bar{N}$ be its opposite and $W(T)$ be the Weyl group. For $f\in \CS_{scusp}(\Fg)$, define
$$I_{Lie}(f,g)=\int_{\Fk} f(g^{-1}kg)dk,\;I_{Lie}(f)=\int_{K\back G}I(f,g)dg.$$
As in the group case, we know that the integral defining $I_{Lie}(f)$ is absolutely convergent. The goal of this section is to prove the following theorem.

\begin{thm}\label{thm Lie algebra}
For all $f\in \CS_{scusp}(\Fg)$, we have
$$I_{Lie}(f)=\frac{1}{c(G,K,\BR)\cdot |W(T)|}\int_{\Ft} D^G(X)^{1/2}\hat{\theta}_f(X)dX$$
where $c(G,K,\BR)=|K\cap T|$ is the constant defined in Definition \ref{defn geom multiplicity group}.
\end{thm}

Since both side of the theorem are continuous on $f$, we only need to prove the identity when the support of $\hat{f}$ is compact modulo conjugation. For the rest of this section, we will assume that the support of $\hat{f}$ is compact modulo conjugation.

The first step is to introduce a sequence of truncation functions $(\kappa_N)_{N>0}$ on $K\back G$. Let
$$T^+=\{t\in T^\circ|\; \alpha(t)\geq 1,\;\forall \alpha\in \Delta\},\;T^-=\{t\in T^\circ|\; \alpha(t)\geq 1,\;\forall \alpha\in \bar{\Delta}\}$$
where $\Delta$ (resp. $\bar{\Delta}$) is the set of positive root with respect to $B=TN$ (resp. $\bar{B}=T\bar{N}$). We have the Iwasawa decomposition $G=BK=\bar{B}K$ and the Cartan decomposition $G=KT^+K=KT^-K$. We fix a sequence $(\kappa_N)_{N>0}$ of smooth functions $\kappa_N:G\rightarrow [0,1]$ satisfying the following three conditions.
\begin{itemize}
\item The function $\kappa_N$ is bi-K-invariant.
\item There exists $C_1>C_2>0$ such that for all $x\in G$, we have
$$\sigma_G(g)\leq C_2N\Rightarrow \kappa_N(g)=1,\;\kappa_N(g)\neq 0\Rightarrow \sigma_G(g)\leq C_1N.$$
\item There exists $C>0$ such that 
$$|\frac{d}{dt}\kappa_N(e^{tX}x)|_{t=0}|\leq C\cdot \Vert X \Vert_{\Fg}$$
for all $x\in G,X\in \Fg$ and $N>0$.
\end{itemize}
It is clear that such functions exist. For $N>0$, define 
$$I_{N,Lie}(f)=\int_{K\back G}\kappa_N(g)I_{Lie}(f,g)dg.$$
We have $I_{Lie}(f)=\lim_{N\rightarrow \infty} I_{N,Lie}(f)$.

\begin{lem}\label{integral transform lem}
For $f\in \CS(\Fg)$, we have
$$I_{N,Lie}(f)=\frac{1}{c(G,K,\BR)\cdot |W(T)|}\int_{\Ft}D^G(X)^{1/2}\int_{T\back G} \kappa_{N,T}(g) \hat{f}(g^{-1}Xg)dgdY$$
where $\kappa_{N,T}(g)=\int_T \kappa_N(tg)dt.$
\end{lem}

\begin{proof}
By the Fourier inversion formula and Lemma \ref{Cartan decomposition}, we have
$$I_{Lie}(f,g)=\int_{\Fk} f(g^{-1}kg)dk=\int_{\Fp} \hat{f}(g^{-1} Xg)dX$$
$$=\frac{1}{|K\cap T|\cdot |W(T)|}\int_{K}\int_{\Ft}D^G(X)^{1/2} \hat{f}(g^{-1}k^{-1}Xkg)dXdk.$$
This implies that
$$I_{N,Lie}(f)=\frac{1}{|K\cap T|\cdot |W(T)|}\int_{K\back G} \kappa_N(g) \int_{K}\int_{\Ft} D^G(X)^{1/2} \hat{f}(g^{-1}kXkg)dYdk $$
$$=\frac{1}{c(G,K,\BR)\cdot |W(T)|}\int_{\Ft}D^G(X)^{1/2}\int_G \kappa_N(g) \hat{f}(g^{-1}Xg)dgdX$$
$$=\frac{1}{c(G,K,\BR)\cdot |W(T)|}\int_{\Ft}D^G(X)^{1/2}\int_{T\back G} \kappa_{N,T}(g) \hat{f}(g^{-1}Xg)dgdX.$$
This proves the lemma.
\end{proof}

\begin{lem}
There exists $C>0$ and $k>0$ such that
$$|\kappa_{N,T}(g)|\leq CN^k$$
for all $N>0$ and $g\in G$.
\end{lem}

\begin{proof}
This is a direct consequence of the definition of $\kappa_N$ and $\kappa_{N,T}$.
\end{proof}

For $l>0$, let $\Ft_{\leq l}$ be the set of $X\in \Ft$ such that  $D^G(X)\leq l$, and let $\Ft_{>l}$ be its complement. We define 
$$I_{N,l}(f)= \int_{\Ft_{>l}}D^G(X)^{1/2}\int_{T\back G}1_{<\log N}(g^{-1}Xg) \kappa_{N,T}(g) \hat{f}(g^{-1}Xg)dgdY$$
where $1_{<\log N}$ is the characteristic function of the set $\{g\in G|\sigma_G(g)<\log N\}$.

\begin{lem}\label{I to I*}
The following statements hold.
\begin{enumerate}
\item There exist $k\in \BN$ and $c>0$ such that $\mid I_{N,Lie}(f)\mid \leq cN^k$ for all $N>0$.
\item There exist $b\geq 1$ and $c>0$ such that $\mid I_{N,Lie}(f)- \frac{1}{c(G,K,\BR)\cdot |W(T)|}I_{N, N^{-b}}(f) \mid\leq cN^{-1}$ for all $N>0$. In particular, we have $$I_{Lie}(f)=\lim_{N\rightarrow \infty}I_{N,Lie}(f)=\frac{1}{c(G,K,\BR)\cdot |W(T)|}\lim_{N\rightarrow \infty}I_{N, N^{-b}}(f).$$
\end{enumerate}
\end{lem}

\begin{proof}
This is a direct consequence of the Lemma above and (1.2.3), (1.2.4) of \cite{B15}.
\end{proof}

The next step is to change the truncation function $\kappa_N$. Let $\Fa=\Fa_T$ and $\Fa^+$ be the positive chamber with respect to the Borel subgroup $B=TN$. For $Y\in \Fa^+$, we have the positive $(G,T)$-orthogonal set $(Y_P)_{P\in \CP(T)}$ as defined in Section 1.9 of \cite{B15}. For $g\in G$, we can define the $(G,T)$-orthogonal set $\CY(g)=(\CY(g)_P)_{P\in \CP(T)}$ with $\CY(g)_P=Y_P-H_{\bar{P}}(g)$. As in (10.10.1) of \cite{B15}, we have the following statement.

\begin{itemize}
\item[(1)] There exists $c>0$ such that for all $g\in G$ and $Y\in \Fa^+$ with
$$\sigma_G(g)\leq c\inf_{\alpha\in \Delta} \delta(Y),$$
the $(G,T)$-orthogonal set $\CY(g)=(\CY(g)_P)_{P\in \CP(T)}$ is positive.
\end{itemize}

On the other hand, if $\CY=(\CY_P)_{P\in \CP(T)}$ is a positive $(G,T)$-orthogonal set and $Q=LU\in \CF(T)$, we will denote by $\sigma_{T}^{Q}(\cdot, \CY)$ (resp. $\tau_{T}^{Q}(\cdot)$) the characteristic function in $\Fa$ of the sum of $\Fa_L$ with the convex hull of the family $(\CY_P)_{P\subset Q}$ (resp. the characteristic function of $\Fa_{T}^{L}+\Fa_{Q}^{+}$). By Lemma 1.8.4(3) of \cite{LW}, we have
\begin{equation}\label{change truncation 1}
\sum_{Q\in \CF(T)}\sigma_{T}^{Q}(\zeta,\CY)\tau_Q(\zeta-\CY_Q)=1,\;\zeta\in \Fa_T.
\end{equation}

For $Y\in \Fa^+$, define the function $\tilde{v}(Y,\cdot)$ on $T\back G$ to be
$$\tilde{v}(Y,g)=\int_T \sigma_{T}^{G}(H_T(t),\CY(g))dt.$$
Set
$$J_{Y,N^{-b}}(f)= \int_{\Ft_{>N^{-b}}}D^G(Y)^{1/2}\int_{T\back G}1_{<\log N}(g^{-1}Xg) \tilde{v}(Y,g) \hat{f}(g^{-1}Xg)dgdY$$
for all $N\geq 1$ and $Y\in \Fa^+$.

\begin{prop}\label{prop change truncation function}
There exists $c_1,c_2>0$ such that
$$|J_{Y,N^{-b}}(f)-I_{N,N^{-b}}(f)|\ll N^{-1}.$$
for all $N\geq 1$ and $Y\in \Fa^+$ that satisfy the following two inequalities
\begin{equation}\label{change truncation 2}
c_1\log(N)\leq \inf_{\alpha\in \Delta}\alpha(Y),\;\sup_{\alpha\in \Delta}\alpha(Y)\leq c_2N.    
\end{equation}
\end{prop}

\begin{proof}
Let $\omega_f$ be a compact subset of $\Ft$ such that $Supp(\hat{f})\cap (\Ft-\omega_f)^G=\emptyset$ (recall that $Supp(\hat{f})$ is compact modulo conjugation), and let $\CA_N$ be the subset of $\omega_f\times T\back G$ consisting of pairs $(X,g)$ such that $D^G(X)>N^{-b}$ and $\sigma_G(g^{-1}Xg)<\log N$. Let $Q=LU\in \CF(T)$ be a proper parabolic subgroup and $\bar{Q}=L\bar{U}$ be its opposite. Let $\CB_N$ be the set of quadruple $(X,l,u,k)\in \Ft\times T\back L \times \bar{U}\times K$ such that $(X,luk)\in \CA_N$. Define
$$\kappa_{N,Q}^{Y}(g)=\int_T \kappa_N(tg)\sigma_{T}^{Q}(H_T(t),\CY(g))\tau_Q(H_M(t)-\CY(g)_Q)dt$$
for all $N>0,\;(X,g)\in \CA_N$ and $Y\in \Fa^+$ that satisfy \eqref{change truncation 2}. The proposition can be proved by the exact same argument as Proposition 10.10.1 of \cite{B15} once we have proved the following lemma which is an analogue of (10.10.10) of \cite{B15} for our case.

\begin{lem}
For $c_1$ large enough, we have
$$|\kappa_{N,Q}^{Y}(luk)-\kappa_{N,Q}^{Y}(lk)|\ll N^{-1}$$
for all $N>1$, $(X,l,u,k)\in \CB_N$ and $Y\in \Fa^+$ satisfying \eqref{change truncation 2}.
\end{lem}

\begin{proof}
By (1.2.2) of \cite{B15}, we know that there exists $c>0$ such that for all $N>1$ and for all $(X,l,u,k)\in \CB_N$, up to translating $l$ by an element of $T$, we have $\sigma_G(l),\sigma_G(u)<c\log(N)$. As a result, we only need to prove the following statement.

\begin{itemize}
\item[(2)] For $c_1$ large enough, we have
$$|\kappa_{N,Q}^{Y}(ug)-\kappa_{N,Q}^{Y}(g)|\ll N^{-1}$$
for all $N>1$, $u\in \bar{U}$, $g\in G$ and $Y\in \Fa^+$ satisfying \eqref{change truncation 2} and the inequalities $\sigma_G(u),\sigma_G(g)<c\log(N)$.
\end{itemize}

We have
$$|\kappa_{N,Q}^{Y}(ug)-\kappa_{N,Q}^{Y}(g)| \leq \int_T |\kappa_N(tug)-\kappa_N(tg)|\sigma_{T}^{Q}(H_T(t),\CY(g))\tau_Q(H_M(t)-\CY(g)_Q)dt.$$
By the definition of $\kappa_N$ and condition on $g,u$, we know that there exists $C>0$ such that
$$\kappa_N(tug)-\kappa_N(tg)\neq 0\Rightarrow \sigma_G(t)\leq CN.$$
Since the volume of the set $\{t\in T|\;\sigma_G(t)\leq CN\}$ is bounded by $C'N^k$ for some $C',k>0$, it is enough to prove the following statement.
\begin{itemize}
\item[(2)] For $c_1$ large enough, for all $t\in T$ satisfying 
$$\sigma_{T}^{Q}(H_T(t),\CY(g))\tau_Q(H_M(t)-\CY(g)_Q)=1,$$
we have 
$$|\kappa_N(tug)-\kappa_N(tg)|\leq N^{-k-1}.$$
\end{itemize}

Fix a large constant $C_3>0$ and let $\Sigma_{Q}^{+}$ be the set of roots of $T$ in $U_Q$. By choosing the constant $c_1$ in \eqref{change truncation 2} large enough, we know that for all $t\in T$ satisfying 
$$\sigma_{T}^{Q}(H_T(t),\CY(g))\tau_Q(H_M(t)-\CY(g)_Q)=1,$$
we have 
$$<\beta,H_T(t)> \geq C_3\log(N),\;\forall \beta\in \Sigma_{Q}^{+}.$$
This implies that for any $C_4>0$, we can choose $c_1$ large enough such that 
$$tut^{-1}\in \exp(B(0,N^{-C_4})),\;B(0,N^{-C_4})=\{X\in \Fg|\;\Vert X\Vert_\Fg < N^{-C_4}\}.$$
Then (2) follows from the last condition of the truncation function $\kappa_N$. This proves the lemma.
\end{proof}
\end{proof}

Now Theorem \ref{thm Lie algebra} follows from the proposition above and the exact same argument as in Section 10.11 of \cite{B15}. We will skip the details here.

\section{The geometric expansion}\label{sec geometric side 1}
In this section, we will prove the geometric expansion of the trace formula. In Section 6.1, we will prove some reductions. In Section 6.2 we will prove the geometric expansion for a special $\theta_K$. In Section 6.3, we will prove the geometric expansion for the general case.

\subsection{Some reductions}
In this subsection we will prove some reductions for the geometric expansion. By our inductional hypothesis, we know that the multiplicity formula and trace formula hold for all the proper Levi subgroups of $G$ and for all the groups $G_t$ where $1\neq t\in G_{ss}$. By the spectral side of the trace formula and the fact that the multiplicity formula holds for all induced representations, we have the following proposition which is an analogue of Proposition 11.5.1 of \cite{B15}.

\begin{prop}\label{geom prop 0}
For any  $\theta_K\in C^{\infty}(K)^K$, the distribution $I(f,\theta_K)$ only depends on the quasi-character $\theta_f$. Moreover, there exists a continuous linear form $J(\cdot,\cdot)$ on $QC(G)\times C^{\infty}(K)^K$ such that
\begin{itemize}
\item $J(\theta_f,\theta_K)=I(f,\theta_K)-I_{geom}(f,\theta_K)$ for all $f\in \CC_{scusp}(G)$ and $\theta_K\in C^{\infty}(K)^K$;
\item $J(\cdot,\cdot)$ is supported on $G_{ell}\times K$.
\end{itemize}
\end{prop}

\begin{proof}
The proof is exactly the same as Proposition 11.5.1 of \cite{B15}. By the spectral expansion proved in the previous section and the fact that the multiplicity formula holds for all induced representations, we can just define $J(\theta,\theta_K)$ to be
$$J(\theta,\theta_K)=\sum_{\pi\in \CX_{ell}(G),\tau\in \CX_{ell}(K)} D(\pi)D(\tau)\theta_K(\tau)(m(\bar{\pi},\bar{\tau})-m_{geom}(\bar{\pi},\bar{\tau}))$$
$$\cdot \int_{\Gamma_{ell}(G)} D^G(x)\theta(x)\theta_\pi(x)dx$$
where $\Gamma_{ell}(G)$ is the set of regular elliptic conjugacy classes of $G$ and $\theta_K(\tau)=\tr(\tau(\theta_K))$. It is clear that $J$ is supported on $G_{ell}\times K$. By (2.7.2), Proposition 4.8.1 and Corollary 5.7.2 of \cite{B15} we know that $J$ is continuous. This proves the proposition.
\end{proof}

Combining the proposition above with Corollary 5.7.2 of \cite{B15}, we have the following corollary.

\begin{cor}
It is enough to prove the geometric expansion when $f$ is compactly supported and $Supp(f)$ is contained in $\Omega_{x}^{G}$ where $\Omega_x$ is a small $G$-good open neighborhood of $G_x$ for some $x\in G_{ell}$.
\end{cor}

\begin{prop}\label{geom prop 1}
The geometric expansion holds when $1\notin Supp(\theta_f)$.
\end{prop}

\begin{proof}
The proposition follows from the same semisimple descent argument as in Section 11.6 of \cite{B15}. To be specific, we may assume that the test function $f$ is of the form as in the Corollary above with $1\neq x\in G_{ell}$. By the proposition above, we may also assume that $f$ comes from a strongly cuspidal function $f_x$ of $G_x$ supported on $\Omega_x$ under the map in Proposition 5.7.1 of \cite{B15}. 

By Proposition 5.7.1 of \cite{B15} and Corollary \ref{geometric multiplicity descent}(1), we know that (note that $|Z_K(x)/K\cap G_x|=|Z_G(x)/G_x|$)
$$I_{geom}(f,\theta_K)=I_{geom,x}((\eta_{x,G})^{-1/2}f_{x,\Omega_x},\eta_{x,K}\theta_K|_{K\cap G_x})$$
where $I_{geom,x}$ is the analogue of $I_{geom}$ for the model $(G_x,K\cap G_x)$. Here we have used the identity $I_{geom,x}(f_x,\theta_K|_{K\cap G_x})=I_{geom,x}({}^zf_x,\theta_K|_{K\cap G_x})$ for all $z\in Z_G(x)$ and $f_x\in \CC_{scusp}(G_x)$ which is due to the facts that $\theta_K$ is invariant under $K$-conjugation and every element in $Z_G(x)$ can be written as $z=z_1z_2$ with $z_1\in Z_K(x),\; z_2\in G_x$ (note that ${}^zf_x(g):=f_x(z^{-1}gz)$). Hence in order to prove the proposition, it is enough to show that 
$$I(f,\theta_K)=I_{x}((\eta_{x,G})^{-1/2}f_{x,\Omega_x},\eta_{x,K}\theta_K|_{K\cap G_x})$$
where $I_{x}$ is the analogue of $I$ for the model $(G_x,K\cap G_x)$.

To prove this, by choosing $\Omega_x$ small, we have
$$I(f,\theta_K)=\int_{K\back G}\int_K {}^gf(k)\theta_K(k)dkdg$$
$$=|Z_K(x)/K\cap G_x|^{-1}\int_{K\back G}\int_{(K\cap G_x)\back K}\int_{K\cap G_x} \eta_{x,K}(k_x) {}^{kg}f(k_x)\theta_K(k_x)dk_xdkdg$$
$$=|Z_K(x)/K\cap G_x|^{-1}\int_{K\back G}\int_{(K\cap G_x)\back K}\int_{K\cap G_x} \eta_{x,K}(k_x)\eta_{x,G}(k_x)^{-1/2} ({}^{kg}f)_{x,\Omega_x}(k_x)\theta_K(k_x)dk_xdkdg$$
$$=|Z_K(x)/K\cap G_x|^{-1}\int_{G_x\back G}\int_{(K\cap G_x)\back G_x}\int_{K\cap G_x} \eta_{x,K}(k_x)\eta_{x,G}(k_x)^{-1/2} ({}^{g}f)_{x,\Omega_x}(g_{x}^{-1}k_xg_x)\theta_K(k_x)dk_xdg_xdg.$$
Here the equation on the second line follows from (3.2.5) of \cite{B15} and the equation on the third line follows from the definition of the function $f_{x,\Omega_x}$ in Section 3.2 of \cite{B15}.

Introduce a function $\alpha$ on $Z_G(x)\back G$ as in Proposition 5.7.1 of \cite{B15}. Then Proposition 5.7.1 of \cite{B15} implies that for all $g\in G_x\back G$, there exists $z\in Z_G(x)$ such that 
$$\int_{(K\cap G_x)\back G_x}\int_{K\cap G_x} \eta_{x,K}(k_x)\eta_{x,G}(k_x)^{-1/2} ({}^{g}f)_{x,\Omega_x}(g_{x}^{-1}k_xg_x)\theta_K(k_x)dk_xdg_x$$
is equal to 
$$\alpha(g)I_{x}((\eta_{x,G})^{-1/2}\cdot {}^zf_x,\eta_{x,K}\theta_K|_{K\cap G_x})=\alpha(g)I_{x}((\eta_{x,G})^{-1/2}f_x,\eta_{x,K}\theta_K|_{K\cap G_x})$$
where the last equation again follows from the fact that every element in $Z_G(x)$ can be written as $z=z_1z_2$ with $z_1\in Z_K(x)$ and $z_2\in G_x$. This implies that
$$I(f,\theta_K)=|Z_K(x)/K\cap G_x|^{-1}\cdot I_{x}((\eta_{x,G})^{-1/2}f_{x,\Omega_x},\eta_{x,K}\theta_K|_{K\cap G_x})\cdot \int_{G_x\back G}\alpha(g)dg$$
$$=I_{x}((\eta_{x,G})^{-1/2}f_{x,\Omega_x},\eta_{x,K}\theta_K|_{K\cap G_x}).$$
This finishes the proof of the proposition.
\end{proof}

\subsection{A special case}

Throughout this subsection, we will use $f$ (resp. $F$) to denote test function on the group (resp. Lie algebra). For $f\in \CC_{scusp}(G)$ and $\theta_K\in C^\infty(K)^K$ (resp. $F\in \CS_{scusp}(\Fg)$ and $\theta_K\in C^\infty(\Fk)^K$), we have defined 
$$I(f,x,\theta_K)=\int_{K} f(x^{-1}kx)\theta_K(k)dk,\;I(f,\theta_K)=\int_{K\back G} I(f,x,\theta_K)dx,$$
$$I_{geom}(f,\theta_K)=m_{geom,G,K}(\theta_f,\theta_K),\;I_{Lie}(F,x,\theta_K)=\int_{\Fk} F(x^{-1}Xx)\theta_K(X)dX,$$
$$I_{Lie}(F,\theta_K)=\int_{K\back G}I_{Lie}(F,x,\theta_K)dx,\;I_{geom,Lie}(F,\theta_K)=m_{geom,G,K,Lie}(\theta_F,\theta_K).$$
For $\theta\in QC(G)$ (resp. $\theta\in QC_c(\Fg)$), let
$$I_{geom}(\theta,\theta_K)=m_{geom,G,K}(\theta,\theta_K),\;I_{geom,Lie}(\theta,\theta_K)=m_{geom,G,K,Lie}(\theta,\theta_K).$$
If $\theta_K$ is the identity function, we will omit $\theta_K$ in the expression, e.g. we will use $I(f)$ to denote $I(f,\theta_K)$ when $\theta_K\equiv 1$.

Let $\omega'\subset \omega\subset \Fg$ be two convex G-excellent open neighborhoods of $0$ and let $\Omega=\exp(\omega),\;\Omega'=\exp(\omega')$. By choosing $\omega'$ small, we can find a function $\theta_{K,0}\in C^\infty(K)^K$ such that
\begin{itemize}
\item $\theta_{K,0}$ is supported on $\Omega\cap K$;
\item the restriction of the function $j_{K}^{1/2}\cdot (\theta_{K,0})_{\omega}$ to $\omega'$ is the identity function. 
\end{itemize}

\begin{lem}\label{lem descent}
For $f\in \CS_{scusp}(\Omega)$ and $\theta_K\in C^\infty(K)^K$ with $Supp(\theta_K)\in \Omega$, we have
$$I(f,\theta_{K})=I_{Lie}(j_{G}^{-1/2}f_{\omega},j_{K}^{1/2}\cdot \theta_{K,\omega}),\;I_{geom}(f,\theta_{K})=I_{Lie,geom}(j_{G}^{-1/2}f_{\omega'},j_{K}^{1/2}\cdot \theta_{K,\omega}).$$
\end{lem}

\begin{proof}
The second identity is a direct consequence of the Corollary \ref{geometric multiplicity descent}. For the first equation, by (3.3.2) of \cite{B15}, we have
$$I(f,x,\theta_{K})=\int_K f(x^{-1}kx)\theta_{K}(k)dk=\int_{\omega\cap \Fk}j_{K}^{1/2}(X) f(x^{-1}\exp(X)x)\theta_{K,\omega}(X)dX$$
$$=\int_{\Fk}j_{K}^{1/2}(X)j_{G}^{-1/2}(X) f_{\omega}(x^{-1}Xx)\theta_{K,\omega}(X)dX=I_{Lie}(j_{G}^{-1/2}f_{\omega'},j_{K}^{1/2}\cdot \theta_{K,\omega},x).$$
This proves the lemma.
\end{proof}

\begin{prop}\label{geom prop 2}
\begin{enumerate}
\item There exists a continuous linear form $J_{Lie}$ on $QC_c(\omega')$ such that 
$$I_{Lie}(F)=J_{Lie}(\theta_F)$$
for all $F\in \CS_{scusp}(\omega')$. 
\item We have $J_{Lie}(\theta)=I_{geom,Lie}(\theta)$ for all $\theta\in QC_c(\omega')$ with $0\notin Supp(f).$
\item We have $J_{Lie}(\theta_{\lambda})=|\lambda|^{\dim(K)} J_{Lie}(\theta)$ for all $\theta\in QC_c(\omega')$ and $\lambda\in \BR^\times$ with $\theta_{\lambda}\in QC_c(\omega')$.
\end{enumerate}
\end{prop}

\begin{proof}
For the first part, by Proposition \ref{geom prop 0}, there exists a linear form $J'$ on $QC(G)$ such that $J'(\theta_f)=I(f,\theta_{K,0})$ for all $f\in \CC_{scusp}(G)$. Using the above lemma, we just need to let $J_{Lie}(\theta)=J'(j_{G}^{1/2}\cdot \theta_\Omega)$ ($\theta_\Omega\in QC_c(\Omega)$ was defined in Section 3.3 of \cite{B15}). The second part follows from Proposition \ref{geom prop 1}.

For the third part, by Proposition 5.6.1 of \cite{B15}, we only need to prove the equation when $\theta=\theta_F$ for some $F\in \CS_{scusp}(\omega')$. Since $\theta_{\lambda}\in QC_c(\omega')$, we may choose $F$ such that $F_\lambda\in \CS_{scusp}(\omega')$ where $F_\lambda(X)=F(\lambda^{-1}X)$. Then we have $I_{Lie}(F)=J_{Lie}(\theta)$ and $I_{Lie}(F_\lambda)=J_{Lie}(\theta_\lambda)$. Then the equation follows from the fact that $I_{Lie}(F_\lambda)=|\lambda|^{\dim(K)} I_{Lie}(F)$.
\end{proof}

\begin{prop}
\begin{enumerate}
\item If $G$ is not split ($\iff \dim(K)>\delta(G)/2$), then  $J_{Lie}(\theta)=I_{geom,Lie}(\theta)$ for all $\theta\in QC_c(\omega')$.
\item If $G$ is split, there exists $c\in \BC$ such that $J_{Lie}(\theta)-I_{geom,Lie}(\theta)=c\cdot c_\theta(0)$ for all $\theta\in QC_c(\omega')$.
\end{enumerate}
\end{prop}

\begin{proof}
If $G$ is not split, by Lemma \ref{lemma quasi-character homogeneous}, there exists $\theta_1,\theta_2\in QC_c(\omega')$ such that $\theta=(M_{\lambda,\dim(K)}-1)\theta_1+\theta_2$ with $0\notin Supp(\theta_2)$. By the corollary above, we have $J_{Lie}(\theta_2)=I_{geom,Lie}(\theta_2)$ and $J_{Lie}((M_{\lambda,\dim(K)}-1)\theta_1)=0$. By Corollary \ref{geometric multiplicity descent}, we have $I_{geom,Lie}((M_{\lambda,\dim(K)}-1)\theta_1)=0$. This proves the first part.

If $G$ is split, by the same argument as in the non-split case we know that that there exists constants $c_{\CO}$ for $\CO\in Nil_{reg}(\Fg)$ such that
$$J_{Lie}(\theta)-I_{geom,Lie}(\theta)=\sum_{\CO\in  Nil_{reg}(\Fg)}c_\CO\cdot c_{\theta,\CO}(0)$$ 
for all $\theta\in QC_c(\omega')$. For any $\CO_1,\CO_2\in Nil_{reg}(\Fg)$, let $g\in Res_{\BC/\BR}G$ be an element as in Lemma \ref{lemma nilpotent orbit}. Since $g^{-1}Kg=K$, the linear forms $J_{Lie}$ and $I_{geom,Lie}$ are invariant under $g$-conjugation. This implies that $c_{\CO_1}=c_{\CO_2}$. This proves the lemma.
\end{proof}

\begin{prop}
Assume that $G$ is split.
\begin{enumerate}
\item There exists a continuous linear form $J_{Lie}$ on $SQC(\Fg)$ such that 
$$I_{Lie}(F)=J_{Lie}(\theta_F)$$
for all $F\in \CS_{scusp}(\Fg)$. 
\item We have $J_{Lie}(\theta_{\lambda})=|\lambda|^{\dim(K)} J_{Lie}(\theta)$ for all $\lambda\in \BR^\times$ and $\theta\in SQC(\Fg)$.
\item We can extend the linear form $I_{geom,Lie}(\theta)$ on $QC_c(\Fg)$ to $SQC(\Fg)$.
\item We have $J_{Lie}(\theta)=I_{geom,Lie}(\theta)$ for all $\theta\in SQC(\Fg)$ with $0\notin Supp(\theta).$
\item There exists $c\in \BC$ such that $J_{Lie}(\theta)-I_{geom,Lie}(\theta)=c\cdot c_\theta(0)$ for all $\theta\in SQC(\Fg)$.
\end{enumerate}
\end{prop}

\begin{proof}
For the first part, by Theorem \ref{thm Lie algebra}, we just need to let
$$J_{Lie}(\theta)=\frac{1}{c(G,K,\BR)\cdot |W(T)|}\int_{\Ft} D^G(X)^{1/2}\hat{\theta}(X)dX.$$
This also proves the second part (note that $\dim(K)=\delta(G)$). The third part follows from Proposition 4.6.1(ii) of \cite{B15}. For (4), by the same argument as in Proposition 11.7.1(iv) of \cite{B15}, we may assume that $\theta\in QC_c(\omega)$. Then (4) has already been proved in the previous proposition. The last part follows from the the first four parts of the proposition together with the same argument as in the previous proposition. This finishes the proof of the proposition.
\end{proof}

\begin{prop}
Assume that $G$ is split. Then $J_{Lie}(\theta)=I_{geom,Lie}(\theta)$ for all $\theta\in SQC(\Fg)$.
\end{prop}

\begin{proof}
We follow the same argument as in Section 11.9 of \cite{B15}. Let $\theta_0$ and $\theta=\hat{\theta}_0$ be the quasi-character defined in Section 11.9 of \cite{B15}. We have $c_\theta(0)\neq 0$. In order to show $c=0$, it is enough to show that $J_{Lie}(\theta)=I_{geom,Lie}(\theta)$. Since $\theta$ is supported in $\Ft$ and $\Ft\cap \CS_{Lie}(G,K)=\{0\}$, we have
$$I_{geom,Lie}(\theta)=\frac{1}{c(G,K,\BR)}c_\theta(0)=\frac{1}{c(G,K,\BR)\cdot |W(T)|}\int_{\Ft} D^G(X)^{1/2}\theta_0(X)dX$$
where the second equation is just (11.9.4) of \cite{B15}. Combining with Theorem \ref{thm Lie algebra}, we have $J_{Lie}(\theta)=I_{geom,Lie}(\theta)$. This finishes the proof of the proposition.
\end{proof}

\subsection{The proof of the geometric expansion}
In this subsection we will prove the geometric side of the trace formula. By the argument in the previous two subsections, it is enough to prove the geometric expansion
$$I_{Lie}(F,\theta_K)=I_{geom,Lie}(F,\theta_K)$$
on the Lie algebra for $F\in C_{c,scusp}^{\infty}(\Fg)$ and $\theta_K\in C^\infty(\Fk)^K$. We also know that the geometric expansion holds when $F$ and $\theta_K$ satisfy one of the following three conditions 
\begin{itemize}
\item $0\notin Supp(F)$;
\item $0\notin Supp(\theta_K)$;
\item $\theta_K\equiv 1$.
\end{itemize}

Let $\omega\subset \Fg$ be an $G$-excellent open neighborhood of $0$ and let $\Omega=\exp(\omega)$.

\begin{prop}
\begin{enumerate}
\item There exists a continuous linear form $J_{Lie}$ on $QC_c(\omega)\times C^\infty(\omega\cap \Fk)^K$ such that 
$$I_{Lie}(F,\theta_K)=J_{Lie}(\theta_F,\theta_K)$$
for all $F\in \CS_{scusp}(\omega)$ and $\theta_K\in C^\infty(\omega\cap \Fk)^K$. 
\item We have $J_{Lie}(\theta,\theta_K)=I_{geom,Lie}(\theta,\theta_K)$ for all $\theta\in QC_c(\omega)$ and $\theta_K\in C^\infty(\omega\cap \Fk)^K$ with $0\notin Supp(\theta)$ or $0\notin Supp(\theta_K)$.
\item We have $J_{Lie}(\theta_{\lambda},\theta_{K,\lambda})=|\lambda|^{\dim(K)} J_{Lie}(\theta,\theta_K)$ for all $\theta\in QC_c(\omega)$, $\theta_K\in C^\infty(\omega\cap \Fk)^K$ and $\lambda\in \BR^\times$ with $\theta_{\lambda}\in QC_c(\omega)$ and $\theta_{K,\lambda}\in C^\infty(\omega\cap \Fk)^K$.
\end{enumerate}
\end{prop}

\begin{proof}
The proof is the same as Proposition \ref{geom prop 2}, we will skip it here.
\end{proof}

By the homogeneous property of the linear form $J_{Lie}$, we can extend it to all $\theta\in QC_c(\Fg)$ and $\theta_K\in C^\infty(\Fk)^K$. We get the following corollary.

\begin{cor}
\begin{enumerate}
\item There exists a continuous linear form $J_{Lie}$ on $QC_c(\Fg)\times C^\infty(\Fk)^K$ such that 
$$I_{Lie}(F,\theta_K)=J_{Lie}(\theta_F,\theta_K)$$
for all $F\in \CS_{scusp}(\Fg)$ and $\theta_K\in C^\infty(\Fk)^K$ with $Supp(F)$ being compact modulo conjugation. 
\item We have $J_{Lie}(\theta,\theta_K)=I_{geom,Lie}(\theta,\theta_K)$ for all $\theta\in QC_c(\Fg)$ and $\theta_K\in C^\infty(\Fk)^K$ with $0\notin Supp(\theta)$ or $0\notin Supp(\theta_K)$.
\item We have $J_{Lie}(\theta_{\lambda},\theta_{K,\lambda})=|\lambda|^{\dim(K)} J_{Lie}(\theta,\theta_K)$ for all $\theta\in QC_c(\Fg)$, $\theta_K\in C^\infty(\Fk)^K$ and $\lambda\in \BR^\times$.
\end{enumerate}
\end{cor}

Let $V=QC(\Fg)\hat{\otimes} C^\infty(\Fk)^K$ realized as functions on $\Fg_{reg}\times \Fk$ and let $V_c$ be the subspace of $V$ consisting of functions whose support is compact modulo conjugation. We can extend the linear forms $J_{Lie}$ and $I_{geom,Lie}$ to $V_c$. The above Corollary implies that 

\begin{enumerate}
\item $J_{Lie}(\Theta)=I_{geom,Lie}(\Theta)$ for all $\Theta\in V_c$ with $0\notin Supp(\Theta)$;
\item $J_{Lie}(M_{\lambda,\dim(K)}\Theta)=J_{Lie}(\Theta)$ and $I_{geom,Lie}(M_{\lambda,\dim(K)}\Theta)=I_{geom,Lie}(\Theta)$ for all $\Theta\in V_c$ where $M_{\lambda,\dim(K)}\Theta(X,Y)=|\lambda|^{-\dim(K)}\Theta(\lambda^{-1}X,\lambda^{-1}Y)$.
\end{enumerate}

Our goal is to prove 
$$J_{Lie}(\theta,\theta_K)=I_{geom,Lie}(\theta,\theta_K)$$
for all $\theta\in QC_c(\theta)$ and $\theta_K\in C^\infty(\Fk)^K$. Since we have already proved the case when $\theta_K\equiv 1$. We may assume that $\theta_K(0)=0$ and it is compactly supported. By Lemma \ref{homogeneous projective}, we can write $\Theta=\theta\times \theta_K\in V_c$ as $(M_{\lambda,\dim(K)}-1) \Theta_1+\Theta_2$ for some $\Theta_1,\Theta_2\in V_c$ with $0\notin Supp(\Theta_2)$. By (1) and (2) above, we have
$$J_{Lie}(\Theta)=J_{Lie}(\Theta_2)=I_{geom,Lie}(\Theta_2)=I_{geom,Lie}(\Theta).$$
This finishes the proof of the geometric expansion.

\appendix

\section{The proof of Lemma \ref{lemma geometric side convergence}}

Our goal is to prove the following lemma.

\begin{lem}
The integral $\int_{\CS(G,K)} \Delta(t)^{s-1/2} dt$ is absolutely convergent when $\frac{1}{2}>Re(s)>0$. Moreover, there exists $d>0$ such that 
$$\lim_{s\rightarrow 0^+} s^d\int_{\CS(G,K)} \Delta(t)^{s-1/2} dt=0.$$
\end{lem}

By induction, we may assume that the lemma holds for all redutive groups whose dimension is less than $G$. Since $\CS(G,K)$ is compact, it is enough to prove the absolute convergence and the inequality near an element $x\in K$. If $x\neq 1$, this follows from the inductional hypothesis (applied to the model $(G_x,K\cap G_x)$). Hence we only need to prove the case when $x=1$. In other words, we only need to prove the following statement. 

\begin{itemize}
\item[(1)] For all $T\in \CT(G,K)^{\circ}$, the integral $\int_{T^\circ} \Delta(t)^{s-1/2} dt$ is absolutely convergent when $\frac{1}{2}>Re(s)>0$ and there exists $d>0$ such that 
$$\lim_{s\rightarrow 0^+} s^d\int_{T^\circ} \Delta(t)^{s-1/2} dt=0.$$
\end{itemize}

If $T$ is not a maximal torus of $K$, then $G_T$ is not abelian. On the other hand, by the definition of $\CT(G,K)$, we know that $G_T$ is split modulo the center. Let $T'$ be a maximal split torus of $G_T$ that is $\theta$-stable, and let $L_T$ be the centralizer of $T'$ in $G$. Then we know that $L_T$ is a proper $\theta$-stable Levi subgroup of $G$ with $T\subset L$. Then (1) follows from the inductional hypothesis (applied to the model $(L_T,K\cap L_T)$). Hence it remains to prove the following statement.

\begin{itemize}
\item[(2)] Let $T$ be a maximal torus of $K$. Then the integral $\int_{T} \Delta(t)^{s-1/2} dt$ is absolutely convergent when $\frac{1}{2}>Re(s)>0$ and there exists $d>0$ such that 
$$\lim_{s\rightarrow 0^+} s^d\int_{T} \Delta(t)^{s-1/2} dt=0.$$
\end{itemize}

If $T$ is not a maximal torus of $G$ ($\iff$ $G$ does not have any discrete series), we can find a $\theta$-stable proper Levi subgroup $L$ of $G$ such that $T\subset L$. Then (2) follows from the inductional hypothesis (applied to the model $(L,K\cap L)$). From now on, assume that $T$ is a maximal torus of $G$ and let $n=\dim(T)=rank(G)$. We have $\dim(K)\geq \frac{\dim(G)-n}{2}$.

Let $\omega$ be a small convex neighborhood of $0$ in $\Ft$ and let $\omega'\subset \omega$ be a open subset such that 
\begin{itemize}
\item $0\notin \overline{\omega'}$;
\item $\omega-\{0\}=\cup_{i=0}^{\infty} \frac{1}{2^i}\cdot \omega'$.
\end{itemize} 
Here $\overline{\omega'}$ is the closure of $\omega'$ and $\frac{1}{2^i}\cdot \omega'=\{\frac{1}{2^i}X|\; X\in \omega'\}$. For example, we can let $\omega'$ to be any open subset contains $\omega-\frac{1}{2}\cdot \omega$ and is contained in $\omega-\frac{1}{3}\cdot \omega$. It is enough to show that 

\begin{itemize}
\item[(3)] the integral $\int_{\omega} \Delta(X)^{s-1/2} dX$ is absolutely convergent when $\frac{1}{2}>Re(s)>0$ and there exists $d>0$ such that $$\lim_{s\rightarrow 0^+} s^d\int_{\omega} \Delta(X)^{s-1/2} dX=0.$$
\end{itemize}

By induction (applied to $(G_X,G_X\cap K)$ where $X$ is any semisimple element of $\omega'$) we know that 
\begin{itemize}
\item[(4)] the integral $\int_{\omega'} \Delta(X)^{s-1/2} dX$ is absolutely convergent when $\frac{1}{2}>Re(s)>0$ and there exists $d>0$ such that $$\lim_{s\rightarrow 0^+} s^d\int_{\omega'} \Delta(X)^{s-1/2} dX=0.$$
\end{itemize}

For $s>0$, we know that $\Delta(t)^{s-1/2}$ is homogeneous of degree greater or equal to $-n+2ns$. This implies that
$$\int_{\omega} \Delta(X)^{s-1/2} dX\leq \sum_{i=0}^{\infty} \int_{\frac{1}{2^i}\cdot\omega'} \Delta(X)^{s-1/2} dX\leq \sum_{i=0}^{\infty} \frac{1}{2^{2nsi}}\int_{\omega'} \Delta(X)^{s-1/2} dX$$
$$=\frac{2^{ns}}{2^{ns}-1}\cdot \int_{\omega'} \Delta(X)^{s-1/2} dX.$$
Then (3) follows from (4), the inequality above and the fact that $$\lim_{s\rightarrow 0^+} s^2\cdot \frac{2^{ns}}{2^{ns}-1}=0.$$ 
This proves the lemma.

\end{document}